\documentclass[10pt,final]{amsart}

\usepackage{mathrsfs}
\usepackage[dvipsnames]{xcolor}
\usepackage{hyperref}
\hypersetup{colorlinks,breaklinks,
            linkcolor=Blue,urlcolor=Blue,
            anchorcolor=Blue,citecolor=Blue}
\usepackage{euscript}
\usepackage[verbose,letterpaper,tmargin=1in,bmargin=1in,lmargin=1in,rmargin=1in]{geometry}
\usepackage{amssymb}
\usepackage{pstricks}
\usepackage{pst-plot}
\usepackage{calc}

\newtheorem{thm}{Theorem}[section]
\newtheorem{cor}[thm]{Corollary}
\newtheorem{lem}[thm]{Lemma}
\newtheorem{prop}[thm]{Proposition}

\newtheorem{mainthm}{Theorem}

\theoremstyle{definition}
\newtheorem{defn}[thm]{Definition}
\newtheorem{ex}[thm]{Example}
\newtheorem{notation}[thm]{Notation}

\newtheorem{rmk*}[thm]{}

\theoremstyle{remark}
\newtheorem{rmk}[thm]{Remark}

\newcommand{\R}{\mathbb R}
\newcommand{\C}{\mathbb C}

\newcommand{\N}{\mathbb N}

\newcommand{\tr}{\mathrm{tr}}
\newcommand{\Tr}{\mathrm{Tr}}

\newcommand{\mc}[1]{\mathcal{#1}}
\newcommand{\mb}[1]{\mathbb{#1}}

\begin{document}

\title[Asymptotic infinitesimal freeness]{Asymptotic infinitesimal freeness with amalgamation for Haar quantum unitary random matrices}

\begin{abstract}
We consider the limiting distribution of $U_NA_NU_N^*$ and $B_N$ (and more general
expressions), where $A_N$ and $B_N$ are $N \times N$ matrices with entries in a unital
C$^*$-algebra $\mc B$ which have limiting $\mc B$-valued distributions as $N \to \infty$,
and $U_N$ is a $N \times N$ Haar distributed quantum unitary random matrix with entries
independent from $\mc B$.  Under a boundedness assumption, we show that $U_NA_NU_N^*$ and
$B_N$ are asymptotically free with amalgamation over $\mc B$.  Moreover, this also holds
in the stronger infinitesimal sense of Belinschi-Shlyakhtenko.

We provide an example which demonstrates that this example may fail for classical Haar
unitary random matrices when the algebra $\mc B$ is infinite-dimensional.
\end{abstract}

\author[S. Curran]{Stephen Curran}
\address{S.C.: Department of Mathematics, University of California, Berkeley, CA 94720, USA.}
\email{\href{mailto:curransr@math.berkeley.edu}{curransr@math.berkeley.edu}}

\author[R. Speicher]{Roland Speicher$^{(\dag)}$}
\address{R.S.: Department of Mathematics and Statistics, Queen's University, Jeffery Hall, Kingston, Ontario K7L 3N6, Canada.}
\email{\href{mailto:speicher@mast.queensu.ca}{speicher@mast.queensu.ca}}
\thanks{$\dag$ Research supported by a Discovery grant from NSERC}


\maketitle

\section{Introduction}
One of the most important results in free probability theory is Voiculescu's asymptotic
freeness for random matrices \cite{voi4}.  One simple form of this result is the
following.  Let $A_N$ and $B_N$ be (deterministic) $N \times N$ matrices with complex
entries, and suppose that $A_N$ and $B_N$ have limiting distributions as $N \to \infty$
with respect to the normalized trace on $M_N(\C)$.  Let $(U_N)_{N \in \N}$ be a sequence
of $N \times N$ unitary random matrices, distributed according to Haar measure.  Then
$U_NA_NU_N^*$ and $B_N$ are asymptotically freely independent as $N \to \infty$.
Moreover, freeness holds ``up to $O(N^{-2})$'', which can be interpreted as
\textit{infinitesimal freeness} in the sense of Belinschi-Shlyakhtenko \cite{bsh}.

On the other hand, it is becoming increasingly apparent that in free probability, the
roles of the classical groups are played by certain ``free'' quantum groups.  This can
most clearly be seen in the study of quantum distributional symmetries, originating with
the free de Finetti theorem of K\"{o}stler and Speicher \cite{ksp} and further developed
in \cite{cur3,cur4,bcs2}, in which the classical permutation, orthogonal and unitary
groups are replaced by Wang's universal compact quantum groups \cite{wang1,wang2}.  For a
general discussion of the passage from classical groups to free quantum groups, see
\cite{bs1}.

In this paper, we will consider the limiting distribution of $U_NA_NU_N^*$ and $B_N$,
where $A_N$ and $B_N$ are as above, but $U_N$ is now a Haar distributed $N \times N$
\textit{quantum} unitary random matrix, in the sense of Wang \cite{wang1}.  We will show
that asymptotic (infinitesimal) freeness now holds even if the entries of $A_N$ and $B_N$
are allowed to take values in an arbitrary unital C$^*$-algebra $\mc B$:

\begin{mainthm}\label{easythm}
Let $\mc B$ be a unital C$^*$-algebra and let $A_N, B_N \in M_N(\mc B)$ for $N \in \N$.
Assume that there is a finite constant $C$ such that $\|A_N\| \leq C$, $\|B_N\| \leq C$
for all $N \in \N$.  For each $N \in \N$, let $U_N$ be a Haar distributed $N \times N$
quantum unitary random matrix, with entries independent from $\mc B$.
\begin{enumerate}
 \item Suppose that there are linear maps $\mu_A,\mu_B:\mc B\langle t \rangle \to \mc B$ such that for any
 $b_0,\dotsc,b_k \in \mc B$,
\begin{align*}
\lim_{N \to \infty} \|(\tr_N \otimes \mathrm{id}_{\mc B})[b_0A_Nb_1\dotsb A_Nb_k] - \mu_A[b_0tb_1\dotsb tb_k]\| &= 0\\
\lim_{N \to \infty} \|(\tr_N \otimes \mathrm{id}_{\mc B})[b_0B_Nb_1\dotsb B_Nb_k] -
\mu_B[b_0tb_1\dotsb tb_k]\| & = 0,
\end{align*}
where $\tr_N$ denotes the normalized trace on $M_N(\C)$.  Then $U_NA_NU_N^*$ and $B_N$
are asymptotically free with amalgamation over $\mc B$.

\item Suppose that in addition, the limits
\begin{align*}
& \lim_{N \to \infty} {N}\Bigl\{(\tr_N \otimes \mathrm{id}_{\mc B})[b_0A_Nb_1\dotsb A_Nb_k]
 - \mu_A[b_0tb_1\dotsb tb_k]\Bigr\}\\
& \lim_{N \to \infty} {N}\Bigl\{(\tr_N \otimes \mathrm{id}_{\mc B})[b_0B_Nb_1\dotsb
B_Nb_k] - \mu_B[b_0tb_1\dotsb tb_k]\Bigr\}
\end{align*}
converge in norm for any $b_0,\dotsc,b_k \in \mc B$.  Then $U_NA_NU_N^*$ and $B_N$ are
asymptotically infinitesimally free with amalgamation over $\mc B$.
\end{enumerate}

\end{mainthm}

We will present more general asymptotic freeness results in Section 5, in particular
Theorem \ref{easythm} will be a special case of Corollary \ref{qrotate}.

For finite-dimensional $\mc B$, we show in Proposition \ref{finitedim} that classical
Haar unitary random matrices are sufficient to obtain such a result.  However, classical
unitaries are in general insufficient for asymptotic freeness with amalgamation, even
within the class of \textit{approximately finite dimensional} C$^*$-algebras, and so it
is indeed necessary to allow quantum unitary transformations.  We will discuss this
further in the second part of Section 5, see in particular Example \ref{counterexample}
and the remarks which follow.

Our paper is organized as follows:  Section 2 contains notations and preliminaries.  Here
we collect the basic notions from free and infinitesimally free probability and introduce
the quantum unitary group $A_u(N)$.  Section 3 contains some combinatorial results,
related to the ``fattening'' operation on noncrossing partitions, which will be required
in the sequel.  In Section 4 we recall the Weingarten formula from \cite{bc1} for
computing integrals over $A_u(N)$, and prove a new estimate on the entries of the
corresponding Weingarten matrix.  Section 5 contains our main results, and a discussion
of their failure for classical Haar unitaries.

\medskip
\noindent{\textbf{Acknowledgements.}} We would like to thank T. Banica, M. Neufang, and
D. Shlyakhtenko for several useful discussions.  S.C. would like to thank his thesis
advisor, D.-V. Voiculescu, for his continued guidance and support while completing
this project.
\section{Preliminaries and notations}

\begin{rmk*} \textbf{Free probability}.  We begin by recalling the basic notions of noncommutative probability spaces and distributions of random variables.
\end{rmk*}

\begin{defn}\hfill
\begin{enumerate}
 \item A \textit{noncommutative probability space} is a pair $(\mc A, \varphi)$, where $\mc A$ is a unital algebra over $\C$ and $\varphi:\mc A \to \C$ is a linear functional such that $\varphi(1) = 1$.  Elements in a noncommutative probability space will be called \textit{random variables}.
\item A W$^*$-probability space $(M,\tau)$ is a von Neumann algebra $M$ together with a faithful, normal, tracial state $\tau$.
\end{enumerate}

\end{defn}

The \textit{joint distribution} of a family $(x_i)_{i \in I}$ of random variables in a
noncommutative probability space $(\mc A,\varphi)$ is the collection of \textit{joint
moments}
\begin{equation*}
 \varphi(x_{i_1}\dotsb x_{i_k})
\end{equation*}
for $k \in \N$ and $i_1,\dotsc,i_k \in I$.  This is nicely encoded in the linear
functional $\varphi_x:\C\langle t_i| i \in I \rangle \to \C$ determined by
\begin{equation*}
 \varphi_x(p) = \varphi(p(x))
\end{equation*}
for $p \in \C \langle t_i| i \in I \rangle$, where $p(x)$ means of course to replace
$t_i$ by $x_i$ for each $i \in I$.

These definitions have natural ``operator-valued'' extensions given by replacing $\C$ by
a more general algebra of scalars, which we now recall.

\begin{defn}
An \textit{operator-valued probability space} $(\mc A, E:\mc A \to \mc B)$ consists of a
unital algebra $\mc A$, a subalgebra $1 \in \mc B \subset \mc A$, and a conditional
expectation $E:\mc A \to \mc B$, i.e., $E$ is a linear map such that $E[1] = 1$ and
\begin{equation*}
 E[b_1ab_2] = b_1E[a]b_2
\end{equation*}
for all $b_1,b_2 \in \mc B$ and $a \in \mc A$.
\end{defn}

\begin{ex}
Let $\mc B$ be a unital algebra over $\C$, and let $M_n(\mc B) = M_n(\C) \otimes \mc B$
be the algebra of $n \times n$ matrices over $\mc B$, with the natural inclusion of $\mc
B$ as $I_n \otimes \mc B$.  Let $\tr_n = n^{-1}\Tr_n$ denote the normalized trace on
$M_n(\C)$.  Then $(M_n(\mc B),\tr \otimes \mathrm{id}_{\mc B})$ is a $\mc B$-valued
probability space.  Note that if $B=(b_{ij})_{i,j=1}^n \in M_n(\mc B)$,
\begin{equation*}
 (\tr_n \otimes \mathrm{id}_{\mc B})\bigl(B\bigr) = \frac{1}{n}\sum_{i=1}^n b_{ii}.
\end{equation*}

\end{ex}

The \textit{$\mc B$-valued joint distribution} of a family $(x_i)_{i \in I}$ of random
variables in an operator-valued probability space $(\mc A,E:\mc A \to \mc B)$ is the
collection of \textit{$\mc B$-valued joint moments}
\begin{equation*}
 E[b_0x_{i_1}\dotsb x_{i_k}b_k]
\end{equation*}
for $k \in \N$, $i_1,\dotsc,i_k \in I$ and $b_0,\dotsc,b_k \in \mc B$.  Again this is
conveniently encoded in the $\mc B$-linear functional $E_x: \mc B \langle t_i|i \in I
\rangle \to \mc B$ determined by
\begin{equation*}
 E_x[p] = E[p(x)]
\end{equation*}
for $p \in \mc B \langle t_i| i \in I \rangle$, the algebra of noncommutative polynomials
with coefficients in $\mc B$.

\begin{defn}
Let $(\mc A, E:\mc A \to \mc B)$ be an operator-valued probability space, and let $(\mc
A_i)_{i \in I}$ be a collection of subalgebras $\mc B \subset \mc A_i \subset A$.  The
algebras are said to be \textit{free with amalgamation over $\mc B$}, or \textit{freely
independent with respect to $E$}, if
\begin{equation*}
 E[a_1\dotsb a_k] = 0
\end{equation*}
whenever $E[a_j] = 0$ for $1 \leq j \leq k$ and $a_j \in \mc A_{i_j}$ with $i_{j} \neq
i_{j+1}$ for $1 \leq j < k$.

We say that subsets $\Omega_i \subset \mc A$ are free with amalgamation over $\mc B$ if
the subalgebras $\mc A_i$ generated by $\mc B$ and $\Omega_i$ are freely independent with
respect to $E$.
\end{defn}

\begin{rmk}
Voiculescu first defined freeness with amalgamation, and developed its basic theory in
\cite{voi1}.  Freeness with amalgamation also has a rich combinatorial structure,
developed in \cite{sp1}, which we now recall.  For further information on the
combinatorial theory of free probability, the reader is referred to the text \cite{ns}.
\end{rmk}

\begin{defn}\hfill
\begin{enumerate}
\item A \textit{partition} $\pi$ of a set $S$ is a collection of disjoint, non-empty sets $V_1,\dotsc,V_r$ such that $V_1 \cup \dotsb \cup V_r = S$.  $V_1,\dotsc,V_r$ are called the \textit{blocks} of $\pi$, and we set $|\pi| = r$. If $s,t \in S$ are in the same block of $\pi$, we write $s \sim_\pi t$. The collection of partitions of $S$ will be denoted $\mc P(S)$, or in the case that $S =\{1,\dotsc,k\}$ by $\mc P(k)$.
\item Given $\pi,\sigma \in \mc P(S)$, we say that $\pi \leq \sigma$ if each
block of $\pi$ is contained in a block of $\sigma$.
There is a least element of $\mc P(S)$ which is larger than both $\pi$ and $\sigma$,
which we denote by $\pi \vee \sigma$.
\item If $S$ is ordered, we say that $\pi \in \mc P(S)$ is \textit{non-crossing} if whenever $V,W$ are blocks of $\pi$ and $s_1 < t_1 < s_2 < t_2$ are such that $s_1,s_2 \in V$ and $t_1,t_2 \in W$, then $V = W$.  The non-crossing partitions can also be defined recursively, a partition $\pi \in \mc P(S)$ is non-crossing if and only if it has a block $V$ which is an interval, such that $\pi \setminus V$ is a non-crossing partition of $S \setminus V$.  The set of non-crossing partitions of $S$ is denoted by $NC(S)$, or by $NC(k)$ in the case that $S = \{1,\dotsc,k\}$.
\item Given $\pi,\sigma \in NC(S)$, the join $\pi \vee \sigma$ taken in $\mc P(S)$ may not be non-crossing.  However, there is a least element of $NC(S)$ which is larger than $\pi$ and $\sigma$, which we will denote by $\pi \vee_{nc} \sigma$.  Note that in this paper we will always use $\pi \vee \sigma$ to denote the join in $\mc P(S)$, even when $\pi,\sigma$ are assumed noncrossing.
\item  Given $i_1,\dotsc,i_k$ in some index set $I$, we denote by $\ker \mathbf i$ the element of $\mc P(k)$ whose blocks are the equivalence classes of the relation
\begin{equation*}
 s \sim t \Leftrightarrow i_s= i_t.
\end{equation*}
Note that if $\pi \in \mc P(k)$, then $\pi \leq \ker \mathbf i$ is equivalent to the
condition that whenever $s$ and $t$ are in the same block of $\pi$, $i_s$ must equal
$i_t$.
\item With $0_n$ and $1_n$ we will denote the smallest and largest element, respectively,
in $\mc P(n)$; i.e., $0_n$ has $n$ blocks, each consisting of one element, and $1_n$ has
only one block. Of course, both $0_n$ and $1_n$ are in $NC(n)$.
\end{enumerate}
\end{defn}

\begin{defn} Let $(\mc A, E:\mc A \to \mc B)$ be an operator-valued probability space.
\begin{enumerate}
 \item A \textit{$\mc B$-functional} is a $n$-linear map $\rho:\mc A^n \to \mc B$ such that
\begin{equation*}
 \rho(b_0a_1b_1,a_2b_2,\dotsc,a_nb_n) = b_0\rho(a_1,b_1a_2,\dotsc,b_{n-1}a_n)b_n
\end{equation*}
for all $b_0,\dotsc,b_n \in \mc B$ and $a_1,\dotsc,a_n \in \mc A$.  Equivalently, $\rho$
is a linear map from $\mc A^{\otimes_B n}$ to $\mc B$, where the tensor product is taken
with respect to the obvious $\mc B$-$\mc B$-bimodule structure on $\mc A$.
\item For each $k \in \N$, let $\rho^{(k)}:\mc A^k \to \mc B$ be a $\mc B$-functional.  For $n \in \N$ and $\pi \in NC(n)$, we define a $\mc B$-functional $\rho^{(\pi)}:\mc A^n \to \mc B$ recursively as follows:  If $\pi = 1_n$ is the partition containing only one block, we set $\rho^{(\pi)} = \rho^{(n)}$.  Otherwise let $V = \{l+1,\dotsc,l+s\}$ be an interval of $\pi$ and define
\begin{equation*}
 \rho^{(\pi)}[a_1,\dotsc,a_n] = \rho^{(\pi \setminus V)}[a_1,\dotsc,a_l\rho^{(s)}(a_{l+1},\dotsc,a_{l+s}),a_{l+s+1},\dotsc,a_n]
\end{equation*}
for $a_1,\dotsc,a_n \in \mc A$.
\end{enumerate}
\end{defn}

\begin{ex}
Let $(\mc A,E:\mc A \to \mc B)$ be an operator-valued probability space, and for $k \in
\N$ let $\rho^{(k)}:\mc A^k \to \mc B$ be a $\mc B$-functional as above.  If
 \begin{equation*}
\pi = \{\{1,8,9,10\},\{2,7\},\{3,4,5\}, \{6\}\} \in NC(10),
\end{equation*}
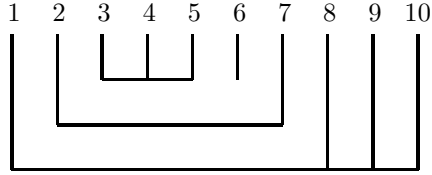
\begin{equation*}
 \setlength{\unitlength}{0.6cm} \begin{picture}(9,4)\thicklines \put(0,0){\line(0,1){3}}
\put(0,0){\line(1,0){9}} \put(9,0){\line(0,1){3}} \put(8,0){\line(0,1){3}}
\put(7,0){\line(0,1){3}} \put(1,1){\line(1,0){5}} \put(1,1){\line(0,1){2}}
\put(6,1){\line(0,1){2}} \put(2,2){\line(1,0){2}} \put(2,2){\line(0,1){1}}
\put(3,2){\line(0,1){1}} \put(4,2){\line(0,1){1}} \put(5,2){\line(0,1){1}}
\put(-0.1,3.3){1} \put(0.9,3.3){2} \put(1.9,3.3){3} \put(2.9,3.3){4} \put(3.9,3.3){5}
\put(4.9,3.3){6} \put(5.9,3.3){7} \put(6.9,3.3){8} \put(7.9,3.3){9} \put(8.7,3.3){10}
\end{picture}
\end{equation*}
then the corresponding $\rho^{(\pi)}$ is given by
\begin{equation*}
 \rho^{(\pi)}[a_1,\dotsc,a_{10}] = \rho^{(4)}(a_1\cdot \rho^{(2)}(a_2\cdot \rho^{(3)}(a_3,a_4,a_5),\rho^{(1)}(a_6)\cdot a_7),a_8,a_9,a_{10}).
\end{equation*}
\end{ex}

\begin{rmk}
Note that if $\mc B$ is commutative, then
\begin{equation*}
 \rho^{(\pi)}[a_1,\dotsc,a_{n}] = \prod_{V \in \pi} \rho(V)[a_1,\dotsc,a_n],
\end{equation*}
where if $V = (i_1 < \dotsb < i_s)$ is a block of $\pi$, we set
\begin{equation*}
 \rho(V)[a_1,\dotsc,a_n] = \rho^{(s)}[a_{i_1},\dotsc,a_{i_s}].
\end{equation*}
\end{rmk}

\begin{defn}
Let $(\mc A, E:\mc A \to \mc B)$ be an operator-valued probability space.
\begin{enumerate}
\item For $k \in \N$, define the \textit{$B$-valued moment functions} $E^{(k)}:\mc A^k \to \mc B$ by
\begin{equation*}
 E^{(k)}[a_1,\dotsc,a_k] = E[a_1\dotsb a_k].
\end{equation*}

\item The \textit{operator-valued free cumulants} $\kappa_E^{(k)}:\mc A^k \to \mc B$ are the $\mc B$-functionals defined by the \textit{moment-cumulant formula}:
\begin{equation*}
 E[a_1\dotsb a_n] = \sum_{\pi \in NC(n)} \kappa_E^{(\pi)}[a_1,\dotsc,a_n]
\end{equation*}
for $n \in \N$ and $a_1,\dotsc,a_n \in \mc A$.
\end{enumerate}

\end{defn}

Note that the right hand side of the moment-cumulant formula above is equal to
$\kappa_E^{(n)}(a_1,\dotsc,a_n)$ plus products of lower order terms and hence can be
solved recursively for $\kappa_E^{(n)}$.  In fact the cumulant functions can be solved
from the moment functions by the following formula from \cite{sp1}: for each $n \in \N$,
$\pi \in NC(n)$ and $a_1,\dotsc,a_n \in \mc A$,
\begin{equation*}
 \kappa_E^{(\pi)}[a_1,\dotsc,a_n] = \sum_{\substack{\sigma \in NC(n)\\ \sigma \leq \pi}} \mu_n(\sigma,\pi)E^{(\sigma)}[a_1,\dotsc,a_n],
\end{equation*}
where $\mu_n$ is the \textit{M\"{o}bius function} on the partially ordered set $NC(n)$.
The M\"{o}bius function is given by the formula
\begin{equation*}
 \mu_n(\sigma,\pi) = \begin{cases} 0, & \sigma \not\leq \pi\\ 1, & \sigma = \pi\\ -1 + \sum_{l \geq 1} (-1)^{l+1}\#\{(\nu_1,\dotsc,\nu_l) \in NC(n)^l: \sigma < \nu_1 < \dotsb < \nu_l < \pi\}, & \sigma < \pi \end{cases}.
\end{equation*}

The key relation between operator-valued free cumulants and freeness with amalgamation is
that freeness can be characterized in terms of the ``vanishing of mixed cumulants''.
\begin{thm}[\cite{sp1}]
Let $(\mc A, E:\mc A \to \mc B)$ be an operator-valued probability space, and let $(\mc
A_i)_{i \in I}$ be a collection of subalgebras $\mc B \subset \mc A_i \subset \mc A$.
Then the family $(\mc A_i)_{i \in I}$ is free with amalgamation over $\mc B$ if and only
if
\begin{equation*}
 \kappa_E^{(\pi)}[a_1,\dotsc,a_n] = 0
\end{equation*}
whenever $a_j \in \mc A_{i_j}$ for $1 \leq j \leq n$ and $\pi \in NC(n)$ is such that
$\pi \not\leq \ker \mathbf i$.
\end{thm}

\begin{rmk*} \textbf{Infinitesimal free probability}.  We will now introduce the notions of operator-valued infinitesimal probability spaces and infinitesimal freeness.  This is a straightforward generalization of the framework of \cite{bsh}, and we refer the reader to that paper for further discussion of infinitesimal freeness and its relation to the type $B$ free independence of Biane, Nica and Goodman \cite{bgn}.  See \cite{fen} for a more combinatorial treatment of infinitesimal freeness.
\end{rmk*}

\begin{defn}\label{infdefn}\hfill
\begin{enumerate}
\item If $\mc B$ is a unital algebra, a \textit{$\mc B$-valued infinitesimal probability space} is a triple $(\mc A, E, E')$ where $\mc A$ is a unital algebra which contains $\mc B$ as a unital subalgebra and $E,E'$ are $\mc B$-linear maps from $\mc A$ to $\mc B$ such that $E[1]= 1$ and $E'[1] = 0$.
\item Let $(\mc A,E,E')$ be a $\mc B$-valued infinitesimal probability space, and
let $(\mc A_i)_{i \in I}$ be a collection of subalgebras $\mc B \subset \mc A_i \subset
\mc A$.  The algebras are said to be \textit{infinitesimally free with amalgamation over
$\mc B$}, or \textit{infinitesimally free with respect to $(E,E')$}, if
\begin{enumerate}
\item $(\mc A_i)_{i \in I}$ are freely independent with respect to $E$.
\item For any $a_1,\dotsc,a_k$ so that $a_{j} \in \mc A_{i_j}$ for $1 \leq j \leq k$
with $i_j \neq i_{j+1}$, we have
\begin{equation*}
 E'\Bigl[\bigl(a_1-E[a_1]\bigr)\dotsb \bigl(a_k-E[a_k]\bigr)\Bigr] =
 \sum_{j=1}^k E\Bigl[\bigl(a_1-E[a_1]\bigr)\dotsb \bigl(E'[a_j]\bigr)\dotsb
 \bigl(a_k-E[a_k]\bigr)\Bigr].
\end{equation*}

\end{enumerate}

We say that subsets $(\Omega_i)_{i \in I}$ are infinitesimally free with amalgamation
over $\mc B$ if the subalgebras $\mc A_i$ generated by $\mc B$ and $\Omega_i$ are
infinitesimally free with respect to $(E,E')$.
\end{enumerate}

\end{defn}

\begin{rmk}
The motivating example is given by a family $(A_i(s))_{i \in I}$ of $\mc B$-valued random
variables for $s > 0$ which are free ``up to $o(s)$'' as $s \to 0$.  This is made precise
in the next proposition. Note that there we make the notion ``free up to $o(s)$'' precise
by comparing the family $(A_i(s))_{i \in I}$ with a family $(a_i(s))_{i \in I}$ which is
free for all $s$. Infinitesimal freeness will then occur at $s=0$ (both for the $A_i$ and
the $a_i$). Since $0$ is not necessarily in $K$, we define the states $E$ and $E'$ on the
free algebra $\mc A:=\mc B \langle A_i| i \in I\rangle$ generated by non-commuting
indeterminates $A_i\hat=A_i(0)\hat = a_i(0)$.
\end{rmk}

\begin{prop}\label{infexample}
Let $\mc B$ be a unital C$^*$-algebra and $K$ a subset of $\R$ for which 0 is an
accumulation point.  Suppose that for each $s\in K$ we have a $\mc B$-valued probability
space $(\mc A(s),E_s:\mc A(s) \to \mc B)$ where $\mc A(s)$ is a unital C$^*$-algebra
which contains $\mc B$ as a unital subalgebra and $E_s$ is contractive. Furthermore,
suppose that, for each $s \in K$, there are variables $(A_i(s))_{i \in I}$ belonging to
$\mc A(s)$ such that the following hold:
\begin{enumerate}
 \item There are $\mc B$-linear maps $E,E':\mc B \langle A_i| i \in I \rangle \to \mc B$
 such that
\begin{align*}
 E[p(A)] &= \lim_{s \to 0} E_s\bigl[p(A(s))\bigr]\\
E'[p(A)] &= \lim_{s \to 0} \frac{1}{s}\Bigl\{E_s[p(A(s))]-E[p]\Bigr\}
\end{align*}
for $p \in \mc B \langle t_i | i \in I \rangle$, where the limits hold in norm.
\item For each $i \in I$,
\begin{equation*}
 \limsup_{s \to 0} \|A_i(s)\| < \infty.
\end{equation*}

\end{enumerate}

Let $I = \bigcup_{j \in J} I_j$ be a partition of $I$.  For $s \in K$, let $(a_i(s))_{i
\in I}$ be a family in some $\mc B$-valued probability space $(\mc C, F:\mc C \to \mc B)$
and suppose that
\begin{enumerate}
\item For any $j \in J$, $p \in \mc B \langle t_i| i \in I_j \rangle$, and $s\in K$,
\begin{equation*}
 E_s[p(A(s))] = F[p(a(s))].
\end{equation*}
\item The sets $(\{a_i(s)| s\in K, i \in I_j\})_{j \in J}$ are free with respect to $F$.
\item
For any $p \in \mc B \langle t_i| i \in I \rangle$ we have
\begin{equation*}
 \bigl\|E_s[p(A(s))] - F[p(a(s))] \bigr\| = o(s) \; \; \; \; (\text{as $s \to 0$}).
\end{equation*}
\end{enumerate}
Then the sets $(\{A_i| i \in I_j\})_{j \in J}\subset \mc B \langle A_i| i \in I \rangle$
are infinitesimally free with respect to $(E,E')$.

\end{prop}

\begin{proof}
Since $E,E'$ only depend on the distribution of the variables $A_i(s)$ up to first order,
it clearly suffices to assume that the sets $(\{A_i(s): i \in I_j\})_{j \in J}$ are
freely independent with respect to $E_s$ for all $s \in K$.  It is then clear that the
sets $(\{A_i: i \in I_j\})_{j \in J}\subset \mc B \langle A_i| i \in I \rangle$ are free
with respect to $E$, so it suffices to show that $E'$ satisfies condition (ii) of
Definition \ref{infdefn}.  Let $j_1 \neq \dotsb \neq j_k$ in $J$ and $p_l \in \mc B
\langle t_i| i \in I_{j_l} \rangle$ for $1 \leq l \leq k$, and consider
\begin{align*}
E'\Bigl[\bigl(p_1(A)-E[p_1(A)]\bigr)&\dotsb \bigl(p_k(A)-E[p_k(A)]\bigr)\Bigr]
 \\
& = \lim_{s \to 0} \frac{1}{s}
 \Bigl\{E_s\Bigl[\bigl(p_1(A(s))-E[p_1(A)]\bigr)\dotsb \bigl(p_k(A(s))-E[p_k(A)]\bigr)\Bigr]\\
 &\qquad\qquad\qquad\qquad\qquad\qquad
 - E\Bigl[\bigl(p_1(A)-E[p_1(A)]\bigr)\dotsb \bigl(p_k(A)-E[p_k(A)]\bigr)\Bigr]\Bigr\}\\
&= \lim_{s \to 0} \frac{1}{s} \Bigl\{E_s\Bigl[\bigl(p_1(A(s))-E[p_1(A)]\bigr)\dotsb
\bigl(p_k(A(s))-E[p_k(A)]\bigr)\Bigr]\Bigr\},
\end{align*}
where we have used freeness with respect to $E$.  Rewrite this expression as
\begin{multline*}
 \lim_{s \to 0} \frac{1}{s} \Bigl\{E_s\Bigl[ \bigl((p_1(A(s))-E_s[p_1(A(s))])+
 (E_s[p_1(A(s))]-E[p_1(A)])\bigr)\\
\dotsb \bigl((p_k(A(s))-E_s[p_k(A(s))])+(E_s[p_k(A(s))]-E[p_k(A)])\bigr)\Bigr]\Bigr\},
\end{multline*}
and consider the terms which appear in the expansion.  First observe that
$\|E_s[p_l(A(s))]-E[p_l(A)]\|$ is $O(s)$ for $1 \leq l \leq k$.  By the boundedness
assumption on the norms of $A_i(s)$, and the contractivity of $E_s$, it follows that
those terms involving more than one expression $(E_s[p_l(A(s))]-E[p_l(A)])$ vanish in the
limit.

The term involving none of these expressions is
\begin{equation*}
 E_s\Bigl[\bigl(p_1(A(s))-E_s[p_1(A(s))]\bigr)\dotsb \bigl(p_k(A(s))-E_s[p_k(A(s))]\bigr)\Bigr]
\end{equation*}
which is zero by freeness.

So we are left to consider only the terms involving one such expression, which gives
\begin{equation*}
 \sum_{l=1}^k \lim_{s \to 0} \frac{1}{s}
 \Bigl\{E_s\Bigl[ \bigl(p_1(A(s)) - E_s[p_1(A(s))]\bigr)
 \dotsb \bigl(E_s[p_l(A(s))]-E[p_l(A)]\bigr)\dotsb \bigl(p_k(A(s))-E_s[p_k(A(s))]\bigr)
 \Bigr]\Bigr\},
\end{equation*}
which again by invoking the boundedness assumptions on $A_i(s)$ and contractivity of
$E_s$, converges to
\begin{equation*}
 \sum_{l=1}^k E\Bigl[\bigl(p_1(A)-E[p_1(A)]\bigr)\dotsb E'[p_l(A)]\dotsb
 \bigl(p_k(A)-E[p_k]\bigr)\Bigr]
\end{equation*}
as desired.
\end{proof}

\begin{rmk*}\textbf{Quantum unitary group.}  We now recall the definition of the quantum unitary group from \cite{wang1}, which is a compact quantum group in the sense of Woronowicz \cite{wor1}.
\end{rmk*}

\begin{defn}
$A_u(n)$ is the universal C$^*$-algebra generated by $\{U_{ij}: 1 \leq i,j \leq n\}$ such
that the matrix $U = (U_{ij}) \in M_n(A_u(n))$ is unitary.  $A_u(n)$ is a C$^*$-Hopf
algebra with comultiplication, counit and antipode given by
\begin{align*}
 \Delta(U_{ij}) &= \sum_{k=1}^n U_{ik} \otimes U_{kj}\\
\epsilon(U_{ij}) &= \delta_{ij}\\
S(U_{ij}) &= U_{ji}^*.
\end{align*}
The existence of these maps is given by the the universal property of $A_u(n)$.
\end{defn}

\begin{rmk}
It is often useful to consider the heuristic formula ``$A_u(n) = C(U_n^+)$'', where
$U_n^+$ is the \textit{free unitary group}.  However in this paper we will stay with the
Hopf algebra notation, which is more convenient for our purposes.
\end{rmk}

\begin{rmk}
A fundamental result of Woronowicz \cite{wor1} guarantees the existence of a unique
\textit{Haar state} $\psi_n:A_u(n) \to \C$ which is left and right invariant in the sense
that
\begin{equation*}
 (\psi_n \otimes \mathrm{id}) \Delta(a) = \psi_n(a) 1_{A_u(n)} = (\mathrm{id} \otimes \psi_n) \Delta(a)
\end{equation*}
for $a \in A_u(n)$.  We will discuss this further in Section 4.
\end{rmk}

Wang also introduced the free product operation on compact quantum groups in
\cite{wang1}.  We will use $A_u(n)^{*\infty}$ to denote the free product of countably
many copies of $A_u(n)$.  The reader is referred to \cite{wang1} for details, the only
properties that we will need are that
\begin{enumerate}
 \item $A_u(n)^{*\infty}$ is generated (as a C$^*$-algebra) by elements $\{U(l)_{ij}: l \in \N, 1 \leq i,j \leq n\}$, such that $U(l) \in M_n(A_u(n)^{*\infty})$ is unitary.
\item The sets $(\{U(l)_{ij}: 1 \leq i,j \leq n\})_{l \in \N}$ are freely independent with respect to the Haar state $\psi_n^{*\infty}$ on $A_u(n)^{*\infty}$, and for each $l \in \N$, $(U(l)_{ij})$ has the same joint distribution in $(A_u(n)^{*\infty},\psi_n^{*\infty})$ as $(U_{ij})$ in $(A_u(n),\psi_n)$.
\end{enumerate}

\section{Some combinatorial results}

In this section we introduce several operations on partitions and prove some basic
results which will be required throughout the remainder of the paper.

\begin{notation} \hfill
\begin{enumerate}
\item Given $\pi \in NC(m)$, we define $\widetilde \pi \in NC_2(2m)$ as follows:  For each block $V = (i_1,\dotsc,i_s)$ of $\pi$, we add to $\widetilde \pi$ the pairings $(2i_1-1,2i_s), (2i_1,2i_2-1),\dotsc,(2i_{s-1},2i_s-1)$.
\item Given $\pi \in NC(m)$, we define $\hat \pi \in NC(2m)$ by
partitioning the $m$-pairs $(1,2),(3,4),\dotsc,(2m-1,2m)$ according to $\pi$.
\item  Given $\pi,\sigma \in \mc P(m)$, we define $\pi \wr \sigma \in \mc P(2m)$ to be the partition obtained by partitioning the odd numbers $\{1,3,\dotsc,2m-1\}$ according to $\pi$ and the even numbers $\{2,4,\dotsc,2m\}$ according to $\sigma$.
\item Given $\pi \in \mc P(m)$, let $\overleftarrow{\pi}$ denote the partition obtained by shifting $k$ to $k-1$ for $1 < k \leq m$ and sending $1$ to $m$, i.e.,
\begin{equation*}
s \sim_{\overleftarrow{\pi}} t \qquad\Longleftrightarrow \qquad(s + 1) \sim_\pi (t+1),
\end{equation*}
where we count modulo $m$ on the right hand side. Likewise we let $\overrightarrow{\pi}$
denote the partition obtained by shifting $k$ to $k+1$ for $1 \leq k < m$ and sending $m$
to $1$.
\end{enumerate}

\end{notation}

\begin{rmk}  The map $\pi \mapsto \widetilde \pi$ is easily seen to be a bijection,
and corresponds to the well-known ``fattening'' operation. The following example shows
this for $\pi=\{\{1,4,5\},\{2,3\},\{6\}\}$.

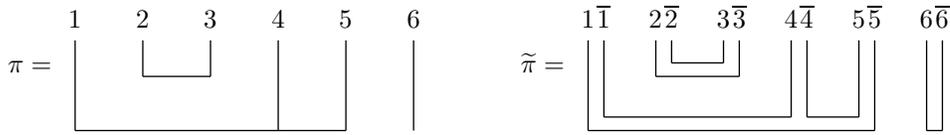
\begin{center}
\begin{pspicture}(0,0)(6,2)
{\psset{xunit=.3cm,yunit=.6cm,linewidth=.5pt}
\psline(0,2)(0,0)\psline(0,0)(12,0) \psline(3,1.2)(6,1.2)
 \psline(12,2)(12,0)
 \psline(3,2)(3,1.2) \psline(6,2)(6,1.2)
\psline(9,2)(9,0)\psline(15,0)(15,2)\uput[u](0,2){1}\uput[u](3,2){2}\uput[u](6,2){3}
\uput[u](9,2){4}\uput[u](12,2){5} \uput[u](15,2){6} \uput[u](-2,1){$\pi=$}}
\end{pspicture}
\qquad
\begin{pspicture}(0,0)(4.8,2)
{\psset{xunit=.3cm,yunit=.6cm,linewidth=.5pt} \psline(0,2)(0,0)\psline(0,0)(12.7,0)
\psline(3,1.2)(6.7,1.2)\psline(3.7,2)(3.7,1.5)
\psline(9.7,2)(9.7,.3)\psline(9.7,.3)(12,.3) \psline(12,2)(12,0.3)
 \psline(.7,2)(.7,.3)\psline(.7,.3)(9,.3)\psline(6.7,1.2)(6.7,2)
\psline(3.7,1.5)(6,1.5) \psline(3,2)(3,1.2) \psline(6,2)(6,1.5)
\psline(9,2)(9,0.3)\psline(15,0)(15,2) \psline(15.7,0)(15.7,2)\psline(15.7,0)(15,0)
\psline(12.7,0)(12.7,2) \uput[u](0,2){1} \uput[u](.7,2){$\overline
1$}\uput[u](3,2){2}\uput[u](3.7,2){$\overline 2$}\uput[u](6,2){3}\uput[u](6.7,2){$\overline 3$}
\uput[u](9,2){4}\uput[u](9.7,2){$\overline
4$}\uput[u](12,2){5}\uput[u](12.7,2){$\overline
5$}\uput[u](15,2){6}\uput[u](15.7,2){$\overline 6$} \uput[u](-2,1){$\widetilde\pi=$}}
\end{pspicture}
\end{center}

There is a simple description of the inverse, it sends $\sigma \in NC_2(2m)$ to the
partition $\tau \in NC(m)$ such that $\sigma \vee \hat 0_m = \hat \tau$, where $\hat 0_m
= \{\{1,2\},\dotsc,\{2m-1,2m\}\}$. Thus we have for $\pi\in NC(m)$
$$\hat \pi=\widetilde \pi\vee \hat 0_m.$$
Note also that $\hat 0_m=\widetilde 0_m$ and that $\hat 1_m=1_{2m}$.

\end{rmk}

\begin{defn}
Let $\pi \in NC(m)$.  The \textit{Kreweras complement} $K(\pi)$ is the largest partition
in $NC(m)$ such that $\pi \wr K(\pi) \in NC(2m)$.
\end{defn}

\begin{ex}
If $\pi = \{\{1,5\},\{2,3,4\},\{6,8\},\{7\}\}$ then $K(\pi) =
\{\{1,4\},\{2\},\{3\},\{5,8\},\{6,7\}\}$, which can be seen follows:
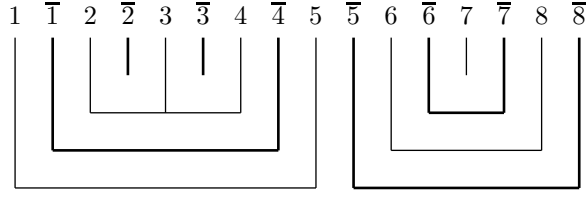
\begin{center}
\begin{pspicture}(0,0)(8,2.5)
{\psset{xunit=.5cm,yunit=1cm,linewidth=.5pt} \psline(0,2)(0,0)\psline(0,0)(8,0)\psline(8,0)(8,2)
\psline(2,2)(2,1)\psline(2,1)(6,1)\psline(6,1)(6,2)\psline(4,1)(4,2)
\psline(10,2)(10,.5)\psline(10,.5)(14,.5)\psline(14,.5)(14,2) \psline(12,2)(12,1.5)}
{\psset{xunit=.5cm,yunit=1cm,linewidth=1pt} \psline(1,2)(1,.5)\psline(1,.5)(7,.5)\psline(7,.5)(7,2)
\psline(3,2)(3,1.5) \psline(5,2)(5,1.5)
\psline(9,2)(9,0)\psline(9,0)(15,0)\psline(15,0)(15,2)
\psline(11,2)(11,1)\psline(11,1)(13,1)\psline(13,1)(13,2) \uput[u](0,2){1}
\uput[u](1,2){$\overline 1$}\uput[u](2,2){2}\uput[u](3,2){$\overline
2$}\uput[u](4,2){3}\uput[u](5,2){$\overline 3$} \uput[u](6,2){4}\uput[u](7,2){$\overline
4$}\uput[u](8,2){5}\uput[u](9,2){$\overline 5$}\uput[u](10,2){6}\uput[u](11,2){$\overline
6$} \uput[u](12,2){7}\uput[u](13,2){$\overline
7$}\uput[u](14,2){8}\uput[u](15,2){$\overline 8$}}
\end{pspicture}
\end{center}

\end{ex}

The following lemma provides the relationship between the Kreweras complement on $NC(m)$
and the map $\pi \mapsto \widetilde \pi$.
\begin{lem}\label{kreweras}
If $\pi \in NC(m)$, then
\begin{equation*}
 \widetilde{K(\pi)} = \overleftarrow{\widetilde \pi}.
\end{equation*}
\end{lem}

\begin{proof}
We will prove this by induction on the number of blocks of $\pi$.  If $\pi = 1_m$ has one
block, the result is trivial from the definitions.

Suppose now that $V = \{l+1,\dotsc,l+s\}$ is a block of $\pi$, $l \geq 1$.  First note
that $\widetilde{\pi}$ is obtained by taking $\widetilde{\pi \setminus V}$ then adding
the pairs $(2l+1,2(l+s)),(2l+2,2l+3),\dotsc,(2(l+s)-2,2(l+s)-1)$.

Observe that $K(\pi)$ is obtained by taking $K(\pi \setminus V)$, adding singletons
$\{l+1\},\dotsc,\{l+s-1\}$, then placing $l+s$ in the block containing $l$.  It follows
that $\widetilde{K(\pi)}$ is the partition obtained by taking $\widetilde{K(\pi \setminus V)}$, which
by induction is $\overleftarrow{\widetilde {\pi \setminus V}}$, then moving the leg
connected to $2l$ to $2(l+s)$ and adding the pairs $(2l,2(l+s)-1)$,
$(2l+1,2l+2),\dotsc,(2(l+s)-3,2(l+s)-2)$.  The result now follows.
\end{proof}

We will also need the following relationship between $\pi \mapsto \widetilde \pi$ and the
Kreweras complement on $NC(2m)$. This is a generalization of the relation
$$K(\hat\pi)=K(\widetilde 0_m\vee \widetilde \pi)=0_m\wr K(\pi)\qquad (\pi\in NC(m)),$$
which is obvious from the definition of $\hat \pi$.
\begin{lem}\label{intertwine}
If $\pi,\sigma \in NC(m)$ and $\sigma \leq \pi$, then $\widetilde \sigma \vee \widetilde
\pi \in NC(2m)$ and
\begin{equation*}
 K(\widetilde \sigma \vee \widetilde \pi) = \sigma \wr K(\pi).
\end{equation*}

\end{lem}

\begin{proof}
We will prove this by induction on the number of blocks of $\pi$.  First suppose that
$\pi = 1_m$, then we have
\begin{equation*}
 \widetilde \sigma \vee \widetilde \pi = \overrightarrow{\overleftarrow{\widetilde \sigma} \vee \overleftarrow{ \widetilde \pi}} = \overrightarrow{\widetilde{ K(\sigma)} \vee \hat 0_m} = \overrightarrow{\widehat{K(\sigma)}}
\end{equation*}
is noncrossing.  Moreover,
\begin{align*}
 K(\widetilde \sigma \vee \widetilde \pi) = K\Bigl(\overrightarrow{\widehat{K(\sigma)}}
 \Bigr) = \overrightarrow{0_m \wr K^2(\sigma)},
\end{align*}
where for the last equality we used the equation for $K(\hat \pi)$ mentioned before the
Lemma \ref{intertwine} and the fact that the Kreweras complement commutes with shifting.
But, by \cite[Exercise 9.23]{ns}, we have that $K^2(\sigma)=\overleftarrow{\sigma}$ and
thus we finally get
\begin{equation*}
K(\widetilde \sigma \vee \widetilde \pi)=\overrightarrow{0_m\wr
\overleftarrow{\sigma}}=\sigma \wr 0_m.
\end{equation*}

Now suppose that $V = \{l+1,\dotsc,l+s\}$, $l \geq 1$ is an interval of $\pi$.  Observe
that $\widetilde \sigma \vee \widetilde \pi$ is the partition obtained by partitioning
$\{1,\dotsc,2l\} \cup \{2(l+s)+1,\dotsc,2m\}$ according to $\widetilde{\sigma \setminus
\sigma|_V} \vee \widetilde{\pi \setminus V}$, and $\{2l+1,\dotsc,2(l+s)\}$ according to
$\widetilde{\sigma|_V} \vee \widetilde{1_V}$.  It follows that $\widetilde \sigma \vee
\widetilde \pi$ is noncrossing and that $K(\widetilde \sigma \vee \widetilde \pi)$ is the
partition obtained by partitioning $\{1,\dotsc,2l\} \cup \{2(l+s)+1,\dotsc,2m\}$
according to $K(\widetilde{\sigma \setminus \sigma|_V} \vee \widetilde{\pi \setminus V})$
and $\{2l+1,\dotsc,2(l+s)\}$ according to $K(\widetilde {\sigma|_V} \vee
\widetilde{1_V})$, then joining the blocks containing $2l$ and $2(l+s)$.  On the other
hand, $K(\pi)$ is equal to the partition obtained by taking $K(\pi \setminus V)$ then
adding $\{l+1\},\dotsc,\{l+s-1\}$ and joining $l+s$ to $l$, and the result now follows by
induction.
\end{proof}

We will need to compare the number of blocks in the join of two partitions before and after fattening.  For this purpose we will use the following \textit{linearization lemma} of Kodiyalam-Sunder \cite{ks}.  Note that the notation $S \mapsto \widetilde S$ used in their paper corresponds to the inverse of the fattening procedure $\pi \mapsto \widetilde \pi$ used here.

\begin{thm}[\cite{ks}]\label{linearization}
Let $\pi, \sigma \in NC(m)$.  Then 
\begin{equation*}
 \bigl|\widetilde \pi \vee \widetilde \sigma \bigr| = m+ 2|\pi \vee \sigma| - |\pi| - |\sigma|.
\end{equation*}
In particular, if $\sigma \leq \pi$ then 
\begin{equation*}
 \bigl|\widetilde \pi \vee \widetilde \sigma \bigr| = m + |\pi| - |\sigma|.
\end{equation*}
\qed
\end{thm}

We now introduce some special classes of noncrossing partitions and prove some basic
results.  These are related to integration on the quantum unitary group via the
Weingarten formula to be discussed in the next section.

\begin{notation} Let $\epsilon_1,\dotsc,\epsilon_{2m} \in \{1,*\}$.
 \begin{enumerate}
\item $NC_h^{\epsilon}(2m)$ denote the set of partitions $\pi \in NC(2m)$ such that each block $V$ of $\pi$ has an even number of elements, and $\epsilon|_V$ is alternating, i.e., $\epsilon|_V = 1*1*\dotsb 1*$ or $*1*1\dotsb *1$.
 \item $NC_2^{\epsilon}(2m)$ will denote the collection of $\pi \in NC_2(2m)$ such that each pair in $\pi$ connects a $1$ with a $*$, i.e.,
\begin{equation*}
 s \sim_\pi t \Rightarrow \epsilon_s \neq \epsilon_t.
\end{equation*}

\item $NC^{\epsilon}(m)$ will denote the collection of $\pi \in NC(m)$ such that $\widetilde \pi \in NC_2^{\epsilon}(m)$.
\end{enumerate}
\end{notation}

\begin{lem} \label{nch}
Let $\epsilon_1,\dotsc,\epsilon_{2m} \in \{1,*\}$.  If $\sigma, \pi \in NC^{\epsilon}(m)$
and $\sigma \leq \pi$, then $\widetilde \sigma \vee \widetilde \pi$ is in
$NC_h^{\epsilon}(2m)$.  Conversely, if $\tau \in NC_h^{\epsilon}(2m)$ then there are
unique $\sigma,\pi \in NC^{\epsilon}(m)$ such that $\sigma \leq \pi$ and $\tau =
\widetilde \sigma \vee \widetilde \pi$.
\end{lem}

\begin{proof}
First suppose that $\tau \in NC_h^{\epsilon}(2m)$.  Since each block of $\tau$ has an
even number of elements, we have $K(\tau) = \sigma \wr K(\pi)$ for some $\sigma,\pi \in
NC(m)$ such that $\sigma \leq \pi$. By Lemma \ref{intertwine} we have $\tau = \widetilde
\sigma \vee \widetilde \pi$, and this clearly determines $\sigma$ and $\pi$ uniquely.  If
$V$ is a block of $\tau$, then $\epsilon|_V$ is alternating and hence $\widetilde
\pi|_V,\widetilde \sigma|_V \in NC_2^{\epsilon}(V)$.  It follows that $\pi,\sigma \in
NC^{\epsilon}(m)$.

Conversely, let $\sigma, \pi \in NC^{\epsilon}(m)$ with $\sigma \leq \pi$.  Let $\hat
\epsilon =
(\epsilon_1,\epsilon_1,\epsilon_2,\epsilon_2,\dotsc,\epsilon_{2m},\epsilon_{2m})$.
Observe that if $\tau \in NC(2m)$, then $\tau \in NC_h^{\epsilon}(2m)$ if and only if
$\widetilde \tau \in NC_2^{\hat \epsilon}(4m)$.

So let $\tau = \widetilde \sigma \vee \widetilde \pi$, we need to show $\widetilde \tau
\in NC_2^{\hat \epsilon}(4m)$.  Now
\begin{align*}
 \overleftarrow{\widetilde \tau} = \widetilde{K(\tau)} = \widetilde{\sigma \wr K(\pi)},
\end{align*}
where we have applied Lemmas \ref{kreweras} and \ref{intertwine}.  In other words,
$\overleftarrow{\widetilde \tau}$ is the partition given by partitioning
$\{1,2,5,6,\dotsc,4m-3,4m-2\}$ according to $\widetilde \sigma$ and
$\{3,4,7,8,\dotsc,4m-1,4m\}$ according to $\widetilde{K(\pi)} = \overleftarrow{\widetilde
\pi}$.  Now since $\sigma,\pi \in NC^\epsilon(m)$, it follows that
$\overleftarrow{\widetilde \tau} \in NC_2^{\overleftarrow{\hat \epsilon}}(4m)$, where
$\overleftarrow{\hat \epsilon} =
(\epsilon_1,\epsilon_2,\epsilon_2,\dotsc,\epsilon_{2m},\epsilon_{2m},\epsilon_1)$, and
hence $\widetilde \tau \in NC_2^{\hat\epsilon}(4m)$.
\end{proof}

\begin{lem}
$NC^{\epsilon}(m)$ is closed under taking intervals in $NC(m)$, i.e., if $\sigma,\pi \in
NC^{\epsilon}(m)$ and $\tau \in NC(m)$ is such that $\sigma < \tau < \pi$, then $\tau \in
NC^{\epsilon}(m)$.
\end{lem}

\begin{proof}
Let $\sigma,\pi \in NC^{\epsilon}(m)$, and $\tau \in NC(m)$ such that $\sigma < \tau <
\pi$.  From the inductive definition of $\widetilde \tau$, to show that $\tau \in
NC^{\epsilon}(m)$ it suffices to consider $\pi = 1_m$.   Now by the previous lemma, we
have $\widetilde \sigma \vee \widetilde{1_m} \in NC_h^{\epsilon}(2m)$.  By Lemma
\ref{kreweras},
\begin{equation*}
 \overleftarrow{\widetilde \sigma \vee \widetilde{1_m}} = \widetilde{K(\sigma)} \vee \hat 0_m = \widehat{K(\sigma)}.
\end{equation*}
Since $\sigma \leq \tau$, we have $\widehat 0_m \leq \widehat{K(\tau)} \leq \widehat
{K(\sigma)}$.  Let $\delta = (\epsilon_{2},\dotsc,\epsilon_{2m},\epsilon_1)$, and suppose
that $\widehat{K(\tau)} \notin NC_h^{\delta}(2m)$.  Let $V$ be a block of $\widehat{
K(\tau)}$, and note that $V$ is of the form $(2i_1-1,2i_1,\dotsc,2i_s-1,2i_s)$ for some
$i_1 < \dotsb < i_s$.  Since $\hat 0_m \in NC_h^{\delta}(2m)$, it follows that there is a
$1 \leq l < s$ with $\delta_{2i_l} = \delta_{2i_{l+1}-1}$.  Now since $\widehat 0_m \leq
\widehat{K(\tau)} \leq \widehat {K(\sigma)}$, it follows that the block $W$ of
$\widehat{K(\sigma)}$ which contains $V$ must have an even number of elements between
$2i_l$ and $2i_{l+1}-1$.  But then $\delta|_{W}$ cannot be alternating, which contradicts
$\widehat{K(\sigma)} \in NC_h^{\delta}(2m)$.

So we have shown that $\widehat{ K(\tau)} \in NC_h^{\delta}(2m)$, and since
\begin{equation*}
 \overrightarrow{\widehat{K(\tau)}} = \overrightarrow{\widetilde{K(\tau)} \vee \hat 0_m} = \widetilde \tau \vee \widetilde 1_m,
\end{equation*}
we have $\widetilde \tau \vee \widetilde {1_m} \in NC_h^{\epsilon}(2m)$.  But then by the
previous lemma, there is a $\gamma \in NC^{\epsilon}(m)$ with $\widetilde \gamma \vee
\widetilde{1_m} = \widetilde \tau \vee \widetilde{1_m}$, and by Lemma \ref{intertwine}
this implies $\tau = \gamma$ is in $NC^\epsilon(m)$ as claimed.
\end{proof}

\section{Integration on the quantum unitary group}

We begin by recalling the \textit{Weingarten formula} from \cite{bc1} for computing
integrals with respect to the Haar state on $A_u(n)$.

Let $\epsilon_1,\dotsc,\epsilon_{2m} \in \{1,*\}$ and define, for $n \in \N$, the
\textit{Gram matrix}
\begin{equation*}
G_{\epsilon n}(\pi,\sigma) = n^{|\pi \vee \sigma|} \;\;\;\;\;\;\; (\pi,\sigma \in
NC_2^{\epsilon}(2m)).
\end{equation*}
It is shown in \cite{bc1} that $G_{\epsilon n}$ is invertible for $n \geq 2$, let
$W_{\epsilon n}$ denote its inverse.

\begin{thm}\cite{bc1}\label{weingarten}
The Haar state on $A_u(n)$ is given by
\begin{align*}
 \psi_n (U_{i_1j_1}^{\epsilon_1}\dotsb U_{i_{2m}j_{2m}}^{\epsilon_{2m}}) &= \sum_{\substack{\pi,\sigma \in NC_2^{\epsilon}(2m)\\ \pi \leq \ker \mathbf i\\ \sigma \leq \ker \mathbf j}} W_{\epsilon n}(\pi,\sigma)\\
\psi_n (U_{i_1j_1}^{\epsilon_1}\dotsb U_{i_{2m+1}j_{2m+1}}^{\epsilon_{2m+1}}) &= 0,
\end{align*}
for $1 \leq i_1,j_1,\dotsc,i_{2m+1},j_{2m+1} \leq n$ and
$\epsilon_1,\dotsc,\epsilon_{2m+1} \in \{1,*\}$.
\end{thm}

\begin{rmk}
Note that the Weingarten formula above is effective for computing integrals of products
of the entries in $U$ and its conjugate $\overline U$, the matrix with $(i,j)$-entry
$U_{ij}^*$.  We will also need to compute integrals of products of entries from $U$ and
its adjoint $U^*$, whose $(i,j)$-entry we denote $(U^*)_{ij}$ to distinguish from the
conjugate $\overline U$.  To do this we will use the following proposition, which allows
us to reduce to the former case.  Note that such a formula clearly fails for the
classical unitary group.
\end{rmk}

\begin{prop} \label{adjint}
Let $1 \leq i_1,i_2,\dotsc,i_{4m} \leq n$ and $\epsilon_1,\dotsc,\epsilon_{2m} \in
\{1,*\}$.  Then
\begin{equation*}
 \psi_n\bigl((U^{\epsilon_1})_{i_1i_2}(U^{\epsilon_2})_{i_3i_4}\dotsb (U^{\epsilon_{2m}})_{i_{4m-1}i_{4m}}\bigr) = \psi_n\bigl(U_{i_1i_2}^{\epsilon_1}U_{i_4i_3}^{\epsilon_2}\dotsb U_{i_{4m}i_{4m-1}}^{\epsilon_{2m}}\bigr).
\end{equation*}
\end{prop}

\begin{proof}
We will use the fact from \cite{ban1} that the joint $*$-distribution of $(U_{ij})_{1
\leq i,j \leq n}$ with respect to $\psi_n$ is the same as that of $(zO_{ij})_{1 \leq i,j
\leq n}$, where $z$ and $(O_{ij})$ are random variables in a $*$-probability space
$(M,\tau)$ such that:
\begin{enumerate}
 \item  $z$ is $*$-freely independent from $\{O_{ij}:1 \leq i,j \leq n\}$.
\item $z$ has a Haar unitary distribution.
\item $(O_{ij})$ are self-adjoint, and have the same joint distribution as the generators of the quantum orthogonal group $A_o(n)$.
\end{enumerate}

The joint distribution of $(O_{ij})$ can also be computed via a Weingarten formula, see
\cite{bc1} for details.  The only fact that we will use is that the joint distribution is
invariant under transposition, i.e., the families $(O_{ij})_{1 \leq i,j \leq n}$ and
$(O_{ji})_{1 \leq i,j \leq n}$ have the same joint distribution.

Now let $\epsilon_1,\dotsc,\epsilon_{2m} \in \{1,*\}$.  Let $A = \{j: \text{$j$ is even
and $\epsilon_j = *$}\} \cup \{j: \text{$j$ is odd and $\epsilon_j = 1$}\}$, and $B =
\{1,\dotsc,2m\} \setminus A$.  Let $1 \leq i_1,j_1,\dotsc,i_{2m},j_{2m} \leq n$.  For $1
\leq k \leq 2m$, define
\begin{align*}
 i'_{k} &= \begin{cases} i_k, & k \in A \\ j_k, & k \in B \end{cases}, & j'_k &= \begin{cases} j_k, & k \in A\\ i_k, & k \in B \end{cases}.
\end{align*}
We claim that
\begin{equation*}
 \psi_n\bigl(U_{i_1j_1}^{\epsilon_1}\dotsb U_{i_{2m}j_{2m}}^{\epsilon_{2m}}\bigr) = \psi_n\bigl(U_{i'_1j'_1}^{\epsilon_1}\dotsb U_{i'_{2m}j'_{2m}}^{\epsilon_{2m}}\bigr),
\end{equation*}
from which the formula in the statement follows immediately.

As discussed above, we have
\begin{equation*}
 \psi_n\bigl(U_{i_1j_1}^{\epsilon_1}\dotsb U_{i_{2m}j_{2m}}^{\epsilon_{2m}}\bigr) = \tau\bigl((zO_{i_1j_1})^{\epsilon_1}\dotsb (zO_{i_{2m}j_{2m}})^{\epsilon_{2m}}\bigr).
\end{equation*}
Note that the expression $(zO_{i_1j_1})^{\epsilon_1}\dotsb
(zO_{i_{2m}j_{2m}})^{\epsilon_{2m}}$ can be written as a product of terms of the form
$zO_{i_kj_k}$ or $O_{i_kj_k}z^*$, depending if $\epsilon_k$ is 1 or $*$.  After rewriting
the expression in this form, let $C$ be the subset of $\{1,\dotsc,4m\}$ consisting of
those indices corresponding to $z$ or $z^*$, and let $D$ be its complement.  Explicitly,
if $\epsilon_k = 1$ then $2k-1$ is in $C$ and $2k$ is in $D$, and if $\epsilon_k= *$ then
$2k$ is in $C$ and $2k-1$ is in $D$.  Given partitions $\alpha,\beta \in NC(2m)$, let
$\Theta(\alpha,\beta) \in P(4m)$ be given by partitioning $C$ according to $\alpha$ and
$D$ according to $\beta$.  By freeness, we have
\begin{equation*}
\tau\bigl((zO_{i_1j_1})^{\epsilon_1}\dotsb (zO_{i_{2m}j_{2m}})^{\epsilon_{2m}}\bigr) =
\sum_{\substack{\alpha,\beta \in NC(2m)\\ \Theta(\alpha,\beta) \in NC(4m)}}
\kappa_{\alpha}[z^{\epsilon_1},\dotsc,z^{\epsilon_{2m}}]\kappa_{\beta}[O_{i_1j_1},\dotsc,O_{i_{2m}j_{2m}}].
\end{equation*}

Now since Haar unitaries are $R$-diagonal, we have in particular that
$\kappa_{\alpha}[z^{\epsilon_1},\dotsc,z^{\epsilon_{2m}}] = 0$ unless each block of
$\alpha$ contains an even number of elements.  So assume that $\alpha$ has this property,
we claim that if $\beta$ is such that $\Theta(\alpha,\beta)$ is noncrossing, then $\beta$
does not join any element of $A$ with an element of $B$.  Indeed, suppose that $\beta$
joins $k_1 < k_2$ and that one of $k_1,k_2$ is in $A$ and the other is in $B$.  If
$k_1,k_2$ have the same parity, then it follows that one of
$\epsilon_{k_1},\epsilon_{k_2}$ is a 1 while the other is a $*$.  Suppose that
$\epsilon_{k_1} = 1,\epsilon_{k_2} = *$, the other case is similar.  Then we have $2k_1$
connected to $2k_2-1$ in $\Theta(\alpha,\beta)$.  Since $\Theta(\alpha,\beta)$ is
noncrossing, $\alpha$ cannot join any element of $\{k_1+1,\dotsc,k_2-1\}$ to an element
outside of this set.  But since this set contains an odd number of elements, we obtain a
contradiction to the choice of $\alpha$.

If $k_1,k_2$ have different parity, then it follows that $\epsilon_{k_1}=\epsilon_{k_2}$.
Suppose that $\epsilon_{k_1}=\epsilon_{k_2} = 1$, the other case is similar.  Then $2k_1$
is connected to $2k_2$ in $\Theta(\alpha,\beta)$.  It follows that $\alpha$ cannot
connect any element of $\{k_1+1,\dotsc,k_2\}$ to an element outside of this set, and
again this set has an odd number of elements which contradicts the choice of $\alpha$.

So the only nonzero terms appearing in the expression above come from $\beta \in NC(2m)$
which split into noncrossing partitions $\pi$ of $A$ and $\sigma$ of $B$.  In this case,
if $A = (a_1 < \dotsb < a_s)$ and $B = (b_1 < \dotsb < b_r)$, we have
\begin{align*}
 \kappa_{\beta}[O_{i_1j_1},\dotsc,O_{i_{2m}j_{2m}}] &= \kappa_{\pi}[O_{i_{a_1}j_{a_1}},\dotsc,O_{i_{a_s}j_{a_s}}]\kappa_{\sigma}[O_{i_{b_1}j_{b_1}},\dotsc,O_{i_{b_r}j_{b_r}}]\\&= \kappa_{\pi}[O_{i_{a_1}j_{a_1}},\dotsc,O_{i_{a_s}j_{a_s}}]\kappa_{\sigma}[O_{j_{b_1}i_{b_1}},\dotsc,O_{j_{b_r}i_{b_r}}]\\&= \kappa_{\beta}[O_{i'_1j'_1},\dotsc,O_{i'_{2m}j'_{2m}}],
\end{align*}
where we have used the invariance of the distribution of $(O_{ij})$ under transposition.

Putting this all together, we have
\begin{align*}
 \psi_n\bigl(U_{i_1j_1}^{\epsilon_1}\dotsb U_{i_{2m}j_{2m}}^{\epsilon_{2m}}\bigr) &= \tau\bigl((zO_{i_1j_1})^{\epsilon_1}\dotsb (zO_{i_{2m}j_{2m}})^{\epsilon_{2m}}\bigr)\\
&= \sum_{\substack{\alpha,\beta \in NC(2m)\\ \Theta(\alpha,\beta) \in NC(4m)}} \kappa_{\alpha}[z^{\epsilon_1},\dotsc,z^{\epsilon_{2m}}]\kappa_{\beta}[O_{i_1j_1},\dotsc,O_{i_{2m}j_{2m}}]\\
&= \sum_{\substack{\alpha,\beta \in NC(2m)\\ \Theta(\alpha,\beta) \in NC(4m)}} \kappa_{\alpha}[z^{\epsilon_1},\dotsc,z^{\epsilon_{2m}}]\kappa_{\beta}[O_{i'_1j'_1},\dotsc,O_{i'_{2m}j'_{2m}}]\\
&= \tau\bigl((zO_{i'_1j'_1})^{\epsilon_1}\dotsb (zO_{i'_{2m}j'_{2m}})^{\epsilon_{2m}}\bigr)\\
&= \psi_n\bigl(U_{i'_1j'_1}^{\epsilon_1}\dotsb U_{i'_{2m}j'_{2m}}^{\epsilon_{2m}}\bigr)
\end{align*}
as desired.

\end{proof}

We can now extend this result to the free product $A_u(n)^{*\infty}$.

\begin{cor}\label{freeadjint}
Let $l_1,\dotsc,l_{2m} \in \N$, $\epsilon_1,\dotsc,\epsilon_{2m} \in \{1,*\}$ and $1 \leq
i_1,j_1,\dotsc,i_{2m},j_{2m} \leq n$.  In $A_u(n)^{*\infty}$, we have
\begin{equation*}
 \psi_n^{*\infty}\bigl((U(l_1)^{\epsilon_1})_{i_1i_2}(U(l_2)^{\epsilon_2})_{i_3i_4}\dotsb (U(l_{2m})^{\epsilon_{2m}})_{i_{4m-1}i_{4m}}\bigr) = \psi_n^{*\infty}\bigl(U(l_1)_{i_1i_2}^{\epsilon_1}U(l_2)_{i_4i_3}^{\epsilon_2}\dotsb U(l_{2m})_{i_{4m}i_{4m-1}}^{\epsilon_{2m}}\bigr).
\end{equation*}

\end{cor}

\begin{proof}
First we claim that in $A_u(n)$, we have
\begin{equation*}
\kappa^{(2m)}[(U^{\epsilon_1})_{i_1i_2},(U^{\epsilon_2})_{i_3i_4},\dotsc,(U^{\epsilon_{2m}})_{i_{4m-1}i_{4m}}]
=
\kappa^{(2m)}[U_{i_1i_2}^{\epsilon_1},U_{i_4i_3}^{\epsilon_2},\dotsc,U_{i_{4m}i_{4m-1}}^{\epsilon_{2m}}].
\end{equation*}
(Note that any cumulant of odd length is zero by Theorem \ref{weingarten}).

Indeed, we have
\begin{multline*}
\kappa^{(2m)}[(U^{\epsilon_1})_{i_1i_2},(U^{\epsilon_2})_{i_3i_4},\dotsc,(U^{\epsilon_{2m}})_{i_{4m-1}i_{4m}}]\\
 = \sum_{\sigma \in NC(2m)} \mu_{2m}(\sigma,1_{2m})\prod_{V \in \sigma} \psi_n(V)[(U^{\epsilon_1})_{i_1i_2},(U^{\epsilon_2})_{i_3i_4},\dotsc,(U^{\epsilon_{2m}})_{i_{4m-1}i_{4m}}].
\end{multline*}
Now it is clear from Theorem \ref{weingarten} that
\begin{equation*}
 \psi_n(V)[(U^{\epsilon_1})_{i_1i_2},(U^{\epsilon_2})_{i_3i_4},\dotsc,(U^{\epsilon_{2m}})_{i_{4m-1}i_{4m}}] = 0
\end{equation*}
unless $V$ has an even number of elements.  So the nonzero terms in the expression above
come from those $\sigma \in NC(2m)$ for which every block has as even number of elements.
For such a $\sigma$, the noncrossing condition implies that each block $V = (l_1 < \dotsb
< l_s)$ must be alternating in parity.  By Proposition \ref{adjint} we have
\begin{align*}
 \psi_n(V)[(U^{\epsilon_1})_{i_1i_2},(U^{\epsilon_2})_{i_3i_4},\dotsc,(U^{\epsilon_{2m}})_{i_{4m-1}i_{4m}}] &= \psi_n\bigl((U^{\epsilon_{l_1}})_{i_{2l_1-1}i_{2l_1}}(U^{\epsilon_{l_2}})_{i_{2l_2-1}i_{2l_2}}\dotsb(U^{\epsilon_{l_s}})_{i_{2l_s-1}i_{2l_s}}\bigr)\\
&= \psi_n\bigl(U^{\epsilon_{l_1}}_{i_{2l_1-1}i_{2l_1}}U^{\epsilon_{l_2}}_{i_{2l_2}i_{2l_2-1}}\dotsb U^{\epsilon_{l_s}}_{i_{2l_s}i_{2l_s-1}}\bigr)\\
&=
\psi_n\bigl(U^{\epsilon_{l_1}}_{i_{2l_1}i_{2l_1-1}}U^{\epsilon_{l_2}}_{i_{2l_2-1}i_{2l_2}}\dotsb
U^{\epsilon_{l_s}}_{i_{2l_s-1}i_{2l_s}}\bigr),
\end{align*}
where the last equation follows from the invariance of the joint $*$-distribution of
$(U_{ij})$ under transposition.  It follows that
\begin{multline*}
\kappa^{(2m)}[(U^{\epsilon_1})_{i_1i_2},(U^{\epsilon_2})_{i_3i_4},\dotsc,(U^{\epsilon_{2m}})_{i_{4m-1}i_{4m}}]\\
 = \sum_{\sigma \in NC(2m)} \mu_{2m}(\sigma,1_{2m})\prod_{V \in \sigma} \psi_n(V)[(U^{\epsilon_1})_{i_1i_2},(U^{\epsilon_2})_{i_3i_4},\dotsc,(U^{\epsilon_{2m}})_{i_{4m-1}i_{4m}}]\\
= \sum_{\sigma \in NC(2m)} \mu_{2m}(\sigma,1_{2m})\prod_{V \in \sigma} \psi_n(V)[U^{\epsilon_1}_{i_1i_2},U^{\epsilon_2}_{i_4i_3},\dotsc,U^{\epsilon_{2m}}_{i_{4m}i_{4m-1}}]\\
=
\kappa^{(2m)}[U_{i_1i_2}^{\epsilon_1},U_{i_4i_3}^{\epsilon_2},\dotsc,U_{i_{4m}i_{4m-1}}^{\epsilon_{2m}}]
\end{multline*}
as claimed.

Now by free independence, in $A_u(n)^{*\infty}$ we have
\begin{multline*}
\psi_n^{*\infty}\bigl((U(l_1)^{\epsilon_1})_{i_1i_2}(U(l_2)^{\epsilon_2})_{i_3i_4}\dotsb (U(l_{2m})^{\epsilon_{2m}})_{i_{4m-1}i_{4m}}\bigr) \\
= \sum_{\substack{\sigma \in NC(2m)\\ \sigma \leq \ker \mathbf l}} \prod_{V \in \sigma}
\kappa(V)[(U(l_1)^{\epsilon_1})_{i_1i_2},(U(l_2)^{\epsilon_2})_{i_3i_4},\dotsc,(U(l_{2m})^{\epsilon_{2m}})_{i_{4m-1}i_{4m}}].
 \end{multline*}
Since $\kappa(V)$ is zero unless $V$ has an even number of elements, the only terms which
contribute to the sum above come again from $\sigma \in NC(2m)$ for which each block has
an even number of elements.  From the previous claim, we have
\begin{equation*}
 \kappa(V)[(U(l_1)^{\epsilon_1})_{i_1i_2},(U(l_2)^{\epsilon_2})_{i_3i_4},\dotsc,(U(l_{2m})^{\epsilon_{2m}})_{i_{4m-1}i_{4m}}] = \kappa(V)[U(l_1)^{\epsilon_1}_{i_1i_2},U(l_2)^{\epsilon_2}_{i_4i_3},\dotsc,U(l_{2m})^{\epsilon_{2m}}_{i_{4m}i_{4m-1}}]
\end{equation*}
for each block $V \in \sigma$, and the result follows immediately.

\end{proof}

\begin{rmk*}
We will now give an estimate on the asymptotic behavior of the entries of $W_{\epsilon
n}$ as $n \to \infty$.  This improves the estimate given in \cite{bc1}.  Note that by
taking $\epsilon = 1*\dotsb 1*$, this estimate also applies to the quantum orthogonal
group, see \cite{bc1}.
\end{rmk*}

\begin{thm}\label{west}
Let $\epsilon_1,\dotsc,\epsilon_{2m} \in \{1,*\}$.  Let $\pi,\sigma \in
NC^{\epsilon}(m)$.  Then
\begin{equation*}
 W_{\epsilon n}(\widetilde \pi,\widetilde \sigma) = O(n^{2|\pi \vee \sigma|-|\pi|-|\sigma|-m}).
\end{equation*}
Moreover,
\begin{equation*}
 n^{m + |\sigma| - |\pi|} W_{\epsilon n}(\widetilde \pi,\widetilde \sigma) = \mu_{m}(\sigma,\pi) + O(n^{-2}),
\end{equation*}
where $\mu_m$ is the M\"{o}bius function on $NC(m)$.

\end{thm}

\begin{proof}
We use a standard method from \cite{col1,colsn1}, further developed in
\cite{bc1,bc2,cur3,bcs2}.

First observe that
\begin{equation*}
 G_{\epsilon n} = \Theta_{\epsilon n}^{1/2}(1 + B_{\epsilon n})\Theta_{\epsilon n}^{1/2},
\end{equation*}
where
\begin{align*}
 \Theta_{\epsilon n}(\pi,\sigma) &= \begin{cases} n^{m}, & \pi = \sigma\\ 0, & \pi \neq \sigma \end{cases}, \\
B_{\epsilon n}(\pi,\sigma) &= \begin{cases} 0, & \pi = \sigma\\ n^{|\pi \vee \sigma| -
m}, & \pi \neq \sigma \end{cases}.
\end{align*}
Note that the entries of $B_{\epsilon n}$ are $O(n^{-1})$, in particular for $n$ large we
have the geometric series expansion
\begin{equation*}
 (1+ B_{\epsilon n})^{-1} = 1 - B_{\epsilon n} + \sum_{l \geq 1} (-1)^{l+1}B_{\epsilon n}^{l+1}.
\end{equation*}
Hence
\begin{equation*}
 W_{\epsilon n}(\widetilde \pi,\widetilde \sigma) = \sum_{l \geq 1} (-1)^{(l+1)} (\Theta_{\epsilon n}^{-1/2}B_{\epsilon n}^{l+1}\Theta_{\epsilon n}^{-1/2})(\widetilde \pi,\widetilde \sigma) + \begin{cases} n^{-m} & \pi = \sigma,\\ -n^{|\widetilde \pi \vee \widetilde \sigma| - 2m} & \pi \neq \sigma. \end{cases}
\end{equation*}
Now for $l \geq 1$, we have
\begin{equation*}
 (\Theta_{\epsilon n}^{-1/2}B_{\epsilon n}^{l+1}\Theta_{\epsilon n}^{-1/2})(\widetilde \pi, \widetilde \sigma) = \sum_{\substack{\nu_1,\dotsc,\nu_l \in NC^{\epsilon}(m)\\ \pi \neq \nu_1 \neq \dotsb \neq \nu_l \neq \sigma}} n^{|\widetilde \pi \vee \widetilde \nu_1| + |\widetilde \nu_1 \vee \widetilde \nu_2| + \dotsb + |\widetilde \nu_l \vee \widetilde \sigma| - (l+2)m}.
\end{equation*}
Now we claim that
\begin{align*}
 |\widetilde \pi \vee \widetilde \nu_1| + \dotsb + |\widetilde \nu_l \vee \widetilde \sigma| &\leq |\widetilde \pi \vee \widetilde \sigma| + |\widetilde \nu_1| + \dotsb + |\widetilde \nu_l|\\
&\leq |\widetilde \pi \vee \widetilde \sigma| + l\cdot m,
\end{align*}
from which (i) follows from the above equation and Theorem \ref{linearization}.

Indeed, the case $l = 1$ follows from the semi-modular condition:
\begin{align*}
|\widetilde \pi \vee \widetilde \nu_1| + |\widetilde \nu_1 \vee \widetilde \sigma| &\leq |(\widetilde \pi \vee \widetilde \nu_1)\vee (\widetilde \nu_1 \vee \widetilde \sigma)| + |(\widetilde \pi \vee \widetilde \nu_1) \wedge (\widetilde \nu_1 \vee \widetilde \sigma)|\\
&\leq |\widetilde \pi \vee \widetilde \sigma| + |\widetilde \nu_1|\\
&= |\widetilde \pi \vee \widetilde \sigma| + m.
\end{align*}
The general case follows easily from induction on $l$.

 For (ii), apply Theorem \ref{linearization} to find that
\begin{align*}
 |\widetilde \pi \vee \widetilde \nu_1| + \dotsb + |\widetilde \nu_l \vee \widetilde \sigma| &= 2(|\nu_1 \vee \nu_2| + \dotsb + |\nu_l \vee \sigma| - |\nu_1| - \dotsb - |\nu_l|) +2|\pi \vee \nu_1| - |\pi| - |\sigma| + (l+1)m\\
&\leq |\pi| - |\sigma| + (l+1)m,
\end{align*}
where equality holds if $\sigma < \nu_l < \dotsb < \nu_1 < \pi$ and otherwise the
difference is at least 2.  It then follows from the equation above that  $n^{m + |\sigma|
- |\pi|}W_{\epsilon n}(\widetilde \pi,\widetilde \sigma)$ is equal to
\begin{equation*}
\begin{cases} 0, & \sigma \not \leq \pi\\ 1, & \pi = \sigma\\ -1 + \sum_{l=1}^\infty (-1)^{l+1} |\{(\nu_1,\dotsc,\nu_l) \in (NC^{\epsilon}(m))^l: \sigma < \nu_l < \dotsb < \nu_1 < \pi\}|, & \sigma < \pi \end{cases},
\end{equation*}
up to $O(n^{-2})$.  Since $NC^{\epsilon}(m)$ is closed under taking intervals in $NC(m)$,
this is equal to $\mu_m(\sigma,\pi)$.

\end{proof}

As a corollary, we can give an estimate on the free cumulants of the generators $U_{ij}$
of $A_u(n)$.  (Note that the cumulants of odd length are all zero since the generators
have an even joint distribution).
\begin{cor}
Let $\epsilon_1,\dotsc,\epsilon_{2m} \in \{1,*\}$ and $i_1,j_1,\dotsc,i_{2m},j_{2m} \in
\N$.  For $\omega \in NC(2m)$, we have for the moment functions
\begin{equation*}
 \psi_n^{(\omega)}[U_{i_1j_1}^{\epsilon_1},\dotsc,U_{i_{2m}j_{2m}}^{\epsilon_{2m}}]
 =  \sum_{\substack{\sigma,\pi \in NC^{\epsilon}(m)\\
 \widetilde \pi \leq \ker \mathbf i \wedge \omega \\ \widetilde \sigma
 \leq \ker \mathbf j \wedge \omega}} n^{|\pi|-|\sigma|-m}(\mu_{m}(\sigma,\pi) +
 O(n^{-2})),
\end{equation*}
and for the cumulant functions
\begin{equation*}
 \kappa^{(\omega)}[U_{i_1j_1}^{\epsilon_1},\dotsc, U_{i_{2m}j_{2m}}^{\epsilon_{2m}}]
 = \sum_{\substack{\pi,\sigma \in NC^{\epsilon}(m)\\
 \widetilde \pi \leq \ker \mathbf i\\ \widetilde \sigma \leq \ker \mathbf j\\
 \widetilde \pi \vee_{nc} \widetilde \sigma = \omega}} n^{|\pi|-|\sigma|-m}
 (\mu_{m}(\pi,\sigma) + O(n^{-2})).
\end{equation*}
\end{cor}

\begin{proof}
First note that
$\psi_n^{(\omega)}[U_{i_1j_1}^{\epsilon_1},\dotsc,U_{i_{2m}j_{2m}}^{\epsilon_{2m}}] = 0$
unless $\omega \in NC_h(2m)$, i.e., unless each block of $\omega$ has an even number of
elements.  So suppose this is the case, then by Lemma \ref{nch} we have $\omega =
\widetilde \alpha \vee \widetilde \beta$ for some $\alpha,\beta \in NC(m)$ with $\alpha
\leq \beta$.  By the Weingarten formula, we have
\begin{equation*}
\psi_n^{(\omega)}[U_{i_1j_1}^{\epsilon_1},\dotsc,U_{i_{2m}j_{2m}}^{\epsilon_{2m}}] =
\sum_{\substack{\pi,\sigma \in NC^{\epsilon}(m)\\ \widetilde \pi \leq \ker \mathbf i
\wedge \omega\\ \widetilde \sigma \leq \ker \mathbf j \wedge \omega}} \prod_{V \in
\omega} W_{\epsilon|_V n}(\widetilde \pi|_V,\widetilde \sigma|_V).
\end{equation*}

Let $V = (l_1 < \dotsb < l_s)$ be a block of $\omega$. In order to apply Theorem
\ref{west} we have to write $\widetilde \pi|_V$ and $\widetilde \sigma|_V$ as $\widetilde
{\pi_V}$ and $\widetilde {\sigma_V}$, respectively, for some $\pi_V,\sigma_V\in NC(\vert
V\vert /2)$. Since $\mu_{\vert V\vert/2}(\sigma_V,\pi_V)=\mu_{\vert V\vert}(\widehat
{\sigma_V},\widehat{\pi_V})$, it suffices to recover the doubled versions $\widehat
{\sigma_V},\widehat{\pi_V}$ from $\widetilde \pi|_V$ and $\widetilde \sigma|_V$. But this
can be achieved as follows.
$$\widehat{\pi_V}=\widetilde{\pi_V}\vee \hat 0_{\vert V\vert/2}=\widetilde \pi|_V\vee
\{(l_1,l_2),\dots,(l_{s-1},l_s)\}.$$ So it remains to write
$\{(l_1,l_2),\dots,(l_{s-1},l_s)\}$ intrinsically in terms of $\omega$.

Recall from Lemma \ref{intertwine} that we have $K(\omega) = \alpha \wr K(\beta)$.  It
follows that for $1 \leq r \leq s$ such that $l_r$ is odd, $\alpha$ has a block whose
least element is $\tfrac{l_r+1}{2}$ and greatest element is $\tfrac{l_{r+1}}{2}$.
Therefore $l_r$ is joined to $l_{r+1}$ in $\widetilde \alpha$.  So if $l_1$ is odd, then
$\widetilde \alpha|_V$ is equal to $\{(l_1,l_2),(l_3,l_4),\dotsc,(l_{s-1},l_s)\}$.
In this case, from Theorem \ref{west} we have
\begin{equation*}
 W_{\epsilon|_V n}(\widetilde \pi|_V, \widetilde \sigma|_V) = n^{\bigl|\widetilde \pi|_V \vee \widetilde \alpha|_V\bigr| -\bigl|\widetilde \sigma|_V \vee \widetilde \alpha|_V\bigr| - |V|/2}(\mu_{|V|}(\widetilde \sigma|_V \vee \widetilde \alpha|_V, \widetilde \pi|_V \vee \widetilde \alpha|_V) + O(n^{-2})).
\end{equation*}
On the other hand, if $l_1$ is even then $\widetilde \alpha|_V =
\{(l_1,l_s),(l_2,l_3),\dotsc,(l_{s-2},l_{s-1})\}$.  In this case we have
\begin{align*}
W_{\epsilon|_V n}(\widetilde \pi|_V, \widetilde \sigma|_V) &= n^{\bigl|\widetilde \sigma|_V \vee \overrightarrow{\widetilde \alpha|_V}\bigr| - \bigl|\widetilde \pi|_V \vee \overrightarrow{\widetilde \alpha|_V}\bigr| - |V|/2}(\mu_{|V|}(\widetilde \pi|_V \vee \overrightarrow{\widetilde \alpha|_V}, \widetilde \sigma|_V \vee \overrightarrow{\widetilde \alpha|_V})+ O(n^{-2}))\\
&= n^{\bigl|\overleftarrow{\widetilde \sigma|_{V}} \vee \widetilde \alpha|_V\bigr| - \bigl|\overleftarrow{\widetilde \pi|_V} \vee \widetilde \alpha|_V\bigr| - |V|/2}(\mu_{|V|}(\overleftarrow{\widetilde \pi|_V} \vee \widetilde \alpha|_V, \overleftarrow{\widetilde \sigma|_V} \vee \widetilde \alpha|_V) + O(n^{-2})),
\end{align*}
where here the arrows act on the legs of $V$.  Since this corresponds, by Lemma \ref{kreweras}, to the Kreweras complement on $NC_{|V|/2}$, we have
\begin{equation*}
\bigl|\overleftarrow{\widetilde \sigma|_{V}} \vee \widetilde \alpha|_V\bigr| = |V|/2+1 - \bigl|\widetilde \sigma|_{V} \vee \widetilde \alpha|_V\bigr|
\end{equation*}
and
\begin{equation*}
\mu_{|V|}(\overleftarrow{\widetilde \pi|_V} \vee \widetilde \alpha|_V, \overleftarrow{\widetilde \sigma|_V} \vee \widetilde \alpha|_V)  = \mu_{|V|}(\widetilde \sigma|_V \vee \widetilde \alpha|_V, \widetilde \pi|_V \vee \widetilde \alpha|_V).
\end{equation*}
So it follows that, as in previous case, we have
\begin{equation*}
W_{\epsilon|_V n}(\widetilde \pi|_V, \widetilde \sigma|_V) = n^{\bigl|\widetilde \pi|_V \vee \widetilde \alpha|_V\bigr| -\bigl|\widetilde \sigma|_V \vee \widetilde \alpha|_V\bigr| - |V|/2}(\mu_{|V|}(\widetilde \sigma|_V \vee \widetilde \alpha|_V, \widetilde \pi|_V \vee \widetilde \alpha|_V) + O(n^{-2})).
\end{equation*}

Therefore,
\begin{align*}
 \psi_n^{(\omega)}[U_{i_1j_1}^{\epsilon_1},\dotsc,U_{i_{2m}j_{2m}}^{\epsilon_{2m}}] &=\sum_{\substack{\sigma,\pi \in NC^{\epsilon}(m)\\ \widetilde \pi \leq \ker \mathbf i \wedge \omega \\ \widetilde \sigma \leq \ker \mathbf j \wedge \omega}} \prod_{V \in \omega} n^{\bigl|\widetilde \pi|_V \vee \widetilde \alpha|_V\bigr| -\bigl|\widetilde \sigma|_V \vee \widetilde \alpha|_V\bigr| - |V|/2}(\mu_{|V|}(\widetilde \sigma|_V \vee \widetilde \alpha|_V, \widetilde \pi|_V \vee \widetilde \alpha|_V) + O(n^{-2})) \\
&= \sum_{\substack{\sigma,\pi \in NC^{\epsilon}(m)\\ \widetilde \pi \leq \ker \mathbf i \wedge \omega \\ \widetilde \sigma \leq \ker \mathbf j \wedge \omega}} n^{|\widetilde \pi \vee \widetilde \alpha|-|\widetilde \sigma \vee \widetilde \alpha|-m}(\mu_{2m}(\widetilde \sigma \vee \widetilde \alpha,\widetilde \pi \vee \widetilde \alpha) + O(n^{-2})),
\end{align*}
where we have used the multiplicativity of the M\"{o}bius function on $NC(2m)$.

Now since $\widetilde \sigma = \widetilde \sigma \vee \widetilde \sigma \leq \widetilde \alpha \vee \widetilde \beta$, taking the Kreweras complement and applying Lemma \ref{intertwine} gives $\alpha \wr K(\beta) \leq \sigma \wr K(\sigma)$.  So we have $\alpha \leq \sigma \leq \beta$.  By Theorem \ref{linearization}, we then have $|\widetilde \sigma \vee \widetilde \alpha| = |\sigma| +m - |\alpha|$.  Also, we have
\begin{align*}
\mu_{2m}(\widetilde \sigma \vee \widetilde \alpha,\widetilde \pi \vee \widetilde \alpha) &= \mu_{2m}(K(\widetilde \pi \vee \widetilde \alpha),K(\widetilde \sigma \vee \widetilde \alpha))\\
&= \mu_{2m}(\alpha \wr K(\pi), \alpha \wr K(\sigma))\\
&= \mu_{m}(K(\pi),K(\sigma))\\
&= \mu_m(\sigma,\pi).
\end{align*}
Plugging this into the equation above, we have
\begin{equation*}
 \psi_n^{(\omega)}[U_{i_1j_1}^{\epsilon_1},\dotsc,U_{i_{2m}j_{2m}}^{\epsilon_{2m}}]
 =  \sum_{\substack{\sigma,\pi \in NC^{\epsilon}(m)\\
 \widetilde \pi \leq \ker \mathbf i \wedge \omega \\ \widetilde \sigma
 \leq \ker \mathbf j \wedge \omega}} n^{|\pi|-|\sigma|-m}(\mu_{m}(\sigma,\pi) + O(n^{-2})).
\end{equation*}

For the cumulant function this gives
\begin{align*}
 \kappa^{(\tau)}[U_{i_1j_1}^{\epsilon_1},\dotsc,U_{i_{2m}j_{2m}}^{\epsilon_{2m}}] &= \sum_{\substack{\omega \in NC(2m)\\ \omega \leq \tau}} \mu_{2m}(\omega,\tau) \psi_n^{(\omega)}[U_{i_1j_1}^{\epsilon_1},\dotsc,U_{i_{2m}j_{2m}}^{\epsilon_{2m}}]
\\
 &= \sum_{\substack{\omega \in NC(2m)\\ \omega \leq \tau}} \mu_{2m}(\omega,\tau) \sum_{\substack{\sigma,\pi \in NC^{\epsilon}(m)\\ \widetilde \pi \leq \ker \mathbf i \wedge \omega \\ \widetilde \sigma \leq \ker \mathbf j \wedge \omega}} n^{|\pi|-|\sigma|-m}(\mu_m(\sigma,\pi) + O(n^{-2}))\\
&= \sum_{\substack{\sigma,\pi \in NC^{\epsilon}(m)\\ \widetilde \pi \leq \ker \mathbf i
\\ \widetilde \sigma \leq \ker \mathbf j }} n^{|\pi|-|\sigma|-m}(\mu_m(\sigma,\pi) +
O(n^{-2})) \sum_{\substack{\omega \in NC(2m)\\ \widetilde \pi \vee_{nc} \widetilde \sigma
\leq \omega \leq \tau}} \mu_{2m}(\omega,\tau).
\end{align*}
Since
\begin{equation*}
 \sum_{\substack{\omega \in NC(2m)\\ \widetilde \pi \vee_{nc} \widetilde \sigma \leq \omega \leq \tau}} \mu_{2m}(\omega,\tau) = \begin{cases} 1, & \widetilde \pi \vee_{nc} \widetilde \sigma = \tau\\ 0, & \text{otherwise}\end{cases},
\end{equation*}
the result follows.

\end{proof}

As a corollary, we can give an estimate on the Haar state on the free product
$A_u(n)^{*\infty}$.

\begin{cor}\label{integrate}
Let $l_1,\dotsc,l_{2m} \in \N$, $\epsilon_1,\dotsc,\epsilon_{2m} \in \{1,*\}$ and
$i_1,j_1,\dotsc,i_{2m},j_{2m} \in \N$.  In $A_u(n)^{*\infty}$, we have
\begin{equation*}
 \psi_n^{*\infty}\Bigl(U(l_1)^{\epsilon_1}_{i_1j_1}\dotsb U(l_{2m})^{\epsilon_{2m}}_{i_{2m}j_{2m}}\Bigr) = \sum_{\substack{\pi,\sigma \in NC^{\epsilon}(m)\\ \widetilde \pi \leq \ker \mathbf i \wedge \ker \mathbf l \\ \widetilde \sigma \leq \ker \mathbf j \wedge \ker \mathbf l}} n^{|\pi| -|\sigma| - m}(\mu_m(\sigma,\pi) + O(n^{-2})).
\end{equation*}

\end{cor}

\begin{proof}
Since the families $(\{U(l)_{ij}\})_{l \in \N}$ are freely independent, we have by the
vanishing of mixed cumulants
\begin{align*}
 \psi_n^{*\infty}\Bigl(U(l_1)^{\epsilon_1}_{i_1j_1}\dotsb U(l_{2m})^{\epsilon_{2m}}_{i_{2m}j_{2m}}\Bigr) = \sum_{\substack{\tau \in NC(2m)\\ \tau \leq \ker \mathbf l}} \kappa^{(\tau)}[U(l_1)^{\epsilon_1}_{i_1j_1},\dotsc,U(l_{2m})^{\epsilon_{2m}}_{i_{2m}j_{2m}}].
\end{align*}
Since the families $(\{U(l)_{ij}\})_{l \in \N}$ are identically distributed, we have
\begin{equation*}
 \kappa^{(\tau)}[U(l_1)^{\epsilon_1}_{i_1j_1},\dotsc,U(l_{2m})^{\epsilon_{2m}}_{i_{2m}j_{2m}}] = \kappa^{(\tau)}[U(1)^{\epsilon_1}_{i_1j_1},\dotsc,U(1)^{\epsilon_{2m}}_{i_{2m}j_{2m}}]
\end{equation*}
for any $\tau \in NC(2m)$ such that $\tau \leq \ker \mathbf l$.  Applying the previous
corollary, we have
\begin{align*}
 \psi_n^{*\infty}\Bigl(U(l_1)^{\epsilon_1}_{i_1j_1}\dotsb U(l_{2m})^{\epsilon_{2m}}_{i_{2m}j_{2m}}\Bigr) &= \sum_{\substack{\tau \in NC(2m)\\ \tau \leq \ker \mathbf l}} \sum_{\substack{\pi,\sigma \in NC^{\epsilon}(m)\\ \widetilde \pi \leq \ker \mathbf i\\ \widetilde \sigma \leq \ker \mathbf j\\ \widetilde \pi \vee_{nc} \widetilde \sigma = \tau}} n^{|\pi|-|\sigma|-m}(\mu_m(\sigma,\pi)+O(n^{-2}))\\
&= \sum_{\substack{\pi,\sigma \in NC^{\epsilon}(m)\\ \widetilde \pi \leq \ker \mathbf i
\wedge \ker \mathbf l\\ \widetilde \sigma \leq \ker \mathbf j \wedge \ker \mathbf l}}
n^{|\pi|-|\sigma|-m}(\mu_m(\sigma,\pi)+O(n^{-2})).
\end{align*}

\end{proof}

\section{Asymptotic freeness results}

\begin{rmk*}  Throughout the first part of this section, the framework will be as follows:  $\mc B$ will be a fixed unital C$^*$-algebra, and $(D_N(i))_{i \in I}$ will be a family of matrices in $M_N(\mc B)$ for $N \in \N$, which is a $\mc B$-valued probability space with conditional expectation $E_N = \tr_N \otimes \mathrm{id}_{\mc B}$.  Consider the free product $A_u(N)^{*\infty}$, generated by the entries in the matrices $(U_N(l))_{l \in \N} \in M_N(A_u(N)^{*\infty})$.  By a family of freely independent Haar quantum unitary random matrices, independent from $\mc B$, we will mean the family $(U_N(l) \otimes 1_{\mc B})_{l \in \N}$ in $M_N(A_u(N)^{*\infty} \otimes \mc B) = M_N(\C) \otimes A_u(N)^{*\infty} \otimes \mc B$, which we will still denote by $(U_N(l))_{l \in \N}$.  We also identify $D_N(i) = D_N(i) \otimes 1_{A_u(N)^{*\infty}}$ for $i \in I$.  We will consider the $\mc B$-valued joint distribution of the family of sets $(\{U_N(1),U_N(1)^*\},\{U_N(2),U_N(2)^*\},\dotsc, \{D_N(i)| i \in I\})$ with respect to the conditional expectation
\begin{equation*}
\psi_N^{*\infty} \otimes E_N = \mathrm{tr}_N \otimes \psi_N^{*\infty} \otimes
\mathrm{id}_{\mc B}.
\end{equation*}
\end{rmk*}

We can now state our main result.

\begin{thm}\label{mainthm}
Let $\mc B$ be a unital C$^*$-algebra, and let $(D_N(i))_{i \in I}$ be a family of
matrices in $M_N(\mc B)$ for $N \in \N$.  Suppose that there is a finite constant $C$
such that $\|D_N(i)\| \leq C$ for all $i \in I$ and $N \in \N$.  Let $(U_N(l))_{l \in
\N}$ be a family of freely independent $N \times N$ Haar quantum unitary random matrices,
independent from $\mc B$.  Let $(u(l),u(l)^*)_{l \in \N}$ and $(d_N(i))_{i \in I, N \in
\N}$ be random variables in a $\mc B$-valued probability space $(\mc A, E: \mc A \to \mc
B)$ such that
\begin{enumerate}
\item $(u(l),u(l)^*)_{l \in \N}$ is free from $(d_N(i))_{i \in I}$ with respect to $E$ for each $N \in \N$..
 \item $(\{u(l),u(l)^*\})_{l \in \N}$ is a free family with respect to $E$, and $u(l)$ is a Haar unitary, independent from $\mc B$ for each $l \in \N$.
\item $(d_N(i))_{i \in I}$ has the same $\mc B$-valued joint distribution with respect to $E$ as $(D_N(i))_{i \in I}$ has with respect to $E_N$.

\end{enumerate}
Then for any polynomials $p_1,\dotsc,p_{2m} \in \mc B \langle t(i)| i \in I \rangle$,
$l_1,\dotsc,l_{2m} \in \N$ and $\epsilon_1,\dotsc,\epsilon_{2m} \in \{1,*\}$,
\begin{equation*}
 \bigl\|(\psi_N^{*\infty} \otimes E_N)[U_N(l_1)^{\epsilon_1}p_1(D_N)\dotsb U_N(l_{2m})^{\epsilon_{2m}}p_{2m}(D_N)] - E[u(l_1)^{\epsilon_1}p_1(d_N)\dotsb u(l_{2m})^{\epsilon_{2m}}p_{2m}(d_N)]\bigr\|
\end{equation*}
is $O(N^{-2})$ as $N \to \infty$.
\end{thm}

Observe that Theorem \ref{mainthm} makes no assumption on the existence of a limiting
distribution for $(D_N(i))_{i \in I}$.  If one assumes also the existence of a limiting
(infinitesimal) $\mc B$-valued joint distribution, then asymptotic (infinitesimal)
freeness follows easily.  We will state this as Theorem \ref{mainthm2} below, let us
first recall the relevant notions.

\begin{defn}
Let $\mc B$ be a unital C$^*$-algebra, and for each $N \in \N$ let $(D_N(i))_{i \in I}$
be a family of noncommutative random variables in a $\mc B$-valued probability space
$(\mc A(N),E_N:\mc A(N) \to \mc B)$.
\begin{enumerate}
\item We say that the joint distribution of $(D_N(i))_{i \in I}$ \textit{converges weakly in norm} if there is a $\mc B$-linear map $E:\mc B \langle D(i)| i \in I \rangle \to \mc B$ such that
\begin{equation*}
 \lim_{N \to \infty} \bigl\|E_N[b_0D_N(i_1)\dotsb D_N(i_k)b_k] - E[b_0D(i_1)\dotsb D(i_k)b_k]\bigr\| = 0
\end{equation*}
for any $i_1,\dotsc,i_k \in I$ and $b_0,\dotsc,b_k \in \mc B$.  If $\mc B$ is a von
Neumann algebra with faithful, normal trace state $\tau$, we say the the joint
distribution of $(D_N(i))_{i \in I}$ \textit{converges weakly in $L^2$} if the equation
above holds with respect to $|\;|_2$.

\item If $I = \bigcup_{j \in J} I_j$ is a partition of $I$, we say that the sequence of sets of random variables $(\{D_N(i)| i \in I_j\})_{j \in J}$ are \textit{asymptotically free with amalgamation over $\mc B$} if the sets $(\{D(i)| i \in I_j\})_{j \in J}$ are freely independent with respect to $E$.

\item We say that the joint distribution of $(D_N(i))_{i \in I}$ \textit{converges infinitesimally in norm} if there is a $\mc B$-linear map $E':\mc B \langle D(i)| i \in I \rangle \to \mc B$ such that
\begin{equation*}
 \lim_{N \to \infty} {N}\Bigl\{E_N[b_0D_N(i_1)\dotsb D_N(i_k)b_k] -
 E[b_0D(i_1)\dotsb D(i_k)b_k]\Bigr\} = E'[b_0D(i_1)\dotsb D(i_k)b_k]
\end{equation*}
with convergence in norm, for any $b_0,\dotsc,b_k \in \mc B$ and $i_1,\dotsc,i_k \in I$.
If $\mc B$ is a von Neumann algebra with faithful, normal trace state $\tau$, we say the
the joint distribution of $(D_N(i))_{i \in I}$ \textit{converges infinitesimally in
$L^2$} if the equation above holds with respect to $|\;|_2$.
\item If $I = \bigcup_{j \in J} I_j$ is a partition of $I$, we say that the sequence of sets of random variables $(\{D_N(i)| i \in I_j\})_{j \in J}$ are \textit{asymptotically infinitesimally free with amalgamation over $\mc B$} if the sets $(\{D(i)| i \in I_j\})_{j \in J}$ are infinitesimally freely independent with respect to $(E,E')$.
\end{enumerate}
\end{defn}

\begin{thm}\label{mainthm2}
Let $\mc B$ be a unital C$^*$-algebra, and let $(D_N(i))_{i \in I}$ be a family of
matrices in $M_N(\mc B)$ for $N \in \N$.  Suppose that there is a finite constant $C$
such that $\|D_N(i)\| \leq C$ for all $i \in I$ and $N \in \N$.  For each $N \in \N$, let
$(U_N(l))_{l \in \N}$ be a family of freely independent $N \times N$ Haar quantum unitary
random matrices, independent from $\mc B$.
\begin{enumerate}
 \item If the joint distribution of $(D_N(i))_{i \in I}$ converges weakly (in norm or in $L^2$ with respect to a faithful trace), then the sets
\begin{equation*}
 (\{U_N(1),U_N(1)^*\},\{U_N(2),U_N(2)^*\},\dotsc,\{D_N(i)|i \in I\})
\end{equation*}
are asymptotically free with amalgamation over $\mc B$ as $N \to \infty$.
\item If the joint distribution of $(D_N(i))_{i \in I}$ converges infinitesimally (in norm or in $L^2$ with respect to a faithful trace), then the sets
\begin{equation*}
 (\{U_N(1),U_N(1)^*\},\{U_N(2),U_N(2)^*\},\dotsc, \{D_N(i)|i \in I\})
\end{equation*}
are asymptotically infinitesimally free with amalgamation over $\mc B$ as $N \to \infty$.
\end{enumerate}

\end{thm}

\begin{rmk*}
Theorem \ref{mainthm2} follows immediately from Theorem \ref{mainthm} and Proposition
\ref{infexample}.  The proof of Theorem \ref{mainthm} will require some preparation, we
begin by computing the limiting distribution appearing in the statement.
\end{rmk*}

\begin{prop}\label{limitform}
Let $(u(l),u(l)^*)_{l \in \N}$ and $(d_N(i))_{i \in I, N \in \N}$ be random variables in
a $\mc B$-valued probability space $(\mc A, E: \mc A \to \mc B)$ such that
\begin{enumerate}
\item $(u(l),u(l)^*)_{l \in \N}$ is free from $(d_N(i))_{i \in I}$ with respect to $E$ for each $N \in \N$..
 \item $(\{u(l),u(l)^*\})_{l \in \N}$ is a free family with respect to $E$, and $u(l)$ is a Haar unitary, independent from $\mc B$ for each $l \in \N$.
\end{enumerate}
Let $a(1),\dotsc,a(2m)$ be in the algebra generated by $\mc B$ and $\{d(i)|i \in I\}$,
and let $l_1,\dotsc,l_{2m} \in \N$ and $\epsilon_1,\dotsc,\epsilon_{2m} \in \{1,*\}$.
Then
\begin{equation*}
E[u(l_1)^{\epsilon_1}a(1)\dotsb u(l_{2m})^{\epsilon_{2m}}a(2m)] =
\sum_{\substack{\pi,\sigma \in NC^{\epsilon}(m)\\ \sigma \leq \pi\\ \widetilde \pi \vee
\widetilde \sigma \leq \ker \mathbf l}} \mu_m(\sigma,\pi) E^{(\sigma \wr
K(\pi))}[a(1),\dotsc,a(2m)].
\end{equation*}

\end{prop}

Note that elements of the form appearing in the statement of the proposition span the
algebra generated by $(u(l),u(l)^*)_{l \in \N}$ and $(d(i))_{i \in I}$, and so this
indeed determines the joint distribution.

\begin{proof}
We have
\begin{equation*}
E[u(l_1)^{\epsilon_1}a(1)\dotsb u(l_{2m})^{\epsilon_{2m}}a(2m)] = \sum_{\alpha \in
NC(4m)}\kappa_E^{\alpha}[u(l_1)^{\epsilon_1},a(1),\dotsc,a(2m)].
\end{equation*}
By freeness, the only non-vanishing cumulants appearing above are those of the form $\tau
\wr \gamma$, where $\tau,\gamma \in NC(2m)$, $\tau \leq \ker \mathbf l$ and $\gamma \leq
K(\tau)$.  So we have
\begin{equation*}
E[u(l_1)^{\epsilon_1}a(1)\dotsb u(l_{2m})^{\epsilon_{2m}}a(2m)] = \sum_{\substack{\tau
\in NC(2m)\\ \tau \leq \ker \mathbf l}} \sum_{\substack{\gamma \in NC(2m)\\ \gamma \leq
K(\tau)}} \kappa_E^{(\tau \wr \gamma)}[u(l_1)^{\epsilon_1},a(1),\dotsc,a(2m)].
\end{equation*}
Since the expectation of any polynomial in $(u(l),u(l)^*)_{l \in \N}$ with complex
coefficients is scalar-valued, it follows that
\begin{align*}
E[u(l_1)^{\epsilon_1}a(1)\dotsb u(l_{2m})^{\epsilon_{2m}}a(2m)] &= \sum_{\substack{\tau \in NC(2m)\\ \tau \leq \ker \mathbf l}} \kappa_E^{(\tau)}[u(l_1)^{\epsilon_1},\dotsc,u(l_{2m})^{\epsilon_{2m}}]\sum_{\substack{\gamma \in NC(2m)\\ \gamma \leq K(\tau)}} \kappa_E^{(\gamma)}[a(1),\dotsc,a(2m)]\\
&= \sum_{\substack{\tau \in NC(2m)\\ \tau \leq \ker \mathbf l}}
\kappa_E^{(\tau)}[u(l_1)^{\epsilon_1},\dotsc,u(l_{2m})^{\epsilon_{2m}}]E^{(K(\tau))}[a(1),\dotsc,a(2m)].
\end{align*}
Since Haar unitaries are $R$-diagonal (\cite[Example 15.4]{ns}), we have
\begin{equation*}
 \kappa_E^{(\tau)}[u(l_1)^{\epsilon_1},\dotsc,u(l_{2m})^{\epsilon_{2m}}]=0
\end{equation*}
unless $\tau \in NC_h^{\epsilon}(2m)$.  By Lemmas \ref{nch} and \ref{intertwine}, we have
\begin{equation*}
E[u(l_1)^{\epsilon_1}a(1)\dotsb u(l_{2m})^{\epsilon_{2m}}a(2m)] =
\sum_{\substack{\pi,\sigma \in NC^{\epsilon}(m)\\ \sigma \leq \pi\\ \widetilde \sigma
\vee \widetilde \pi \leq \ker \mathbf l}} \kappa_E^{(\widetilde \sigma \vee \widetilde
\pi)}[u(l_1)^{\epsilon_1},\dotsc,u(l_{2m})^{\epsilon_{2m}}]E^{(\sigma \wr
K(\pi))}[a(1),\dotsc,a(2m)].
\end{equation*}
So it remains only to show that if $\sigma,\pi \in NC^\epsilon(m)$ and $\sigma \leq \pi$
then
\begin{equation*}
 \mu_m(\sigma,\pi) = \kappa_E^{(\widetilde \sigma \vee \widetilde \pi)}[u(l_1)^{\epsilon_1},\dotsc,u(l_{2m})^{\epsilon_{2m}}].
\end{equation*}
Since the M\"{o}bius function is multiplicative on $NC(m)$, we have
\begin{equation*}
 \mu_m(\sigma,\pi) = \prod_{W \in \pi} \mu_{|W|}(\sigma|_W,1_W),
\end{equation*}
and so it suffices to consider the case $\pi = 1_m$.

By \cite[Proposition 15.1]{ns},
\begin{equation*}
 \kappa_E^{(\widetilde \sigma \vee \widetilde {1_m})}[u(l_1)^{\epsilon_1},\dotsc,u(l_{2m})^{\epsilon_{2m}}] = \prod_{V \in \widetilde \sigma \vee \widetilde {1_m}} (-1)^{|V|/2-1}C_{|V|/2 -1},
\end{equation*}
where $C_n$ is the $n$-th Catalan number.  Since
\begin{equation*}
 \widetilde \sigma \vee \widetilde{1_m} = \overrightarrow{\overleftarrow{\widetilde \sigma} \vee \overleftarrow{\widetilde 1_m}}
= \overrightarrow{\widetilde{K(\sigma)} \vee \widehat{0_m}}= \overrightarrow{\widehat{K(\sigma)}},
\end{equation*}
we have
\begin{equation*}
 \kappa_E^{(\widetilde \sigma \vee \widetilde{1_m})}[u(l_1)^{\epsilon_1},\dotsc,u(l_{2m})^{\epsilon_{2m}}] = \prod_{W \in K(\sigma)} (-1)^{|W|-1}C_{|W| -1}.
\end{equation*}

On the other hand, we have
\begin{align*}
 \mu_m(\sigma,1_m) &= \mu_m(0_m,K(\sigma))\\
&= \prod_{W \in K(\sigma)}\mu_{|W|}(0_W,1_W)\\
&= \prod_{W \in K(\sigma)}(-1)^{|W|-1}C_{|W|-1},
\end{align*}
where we have used the formula for $\mu_m(0_m,1_m)$ from \cite[Proposition 10.15]{ns}.
\end{proof}

\begin{prop}\label{expfunct}
Let $\mc B$ be a unital algebra, $A(1),\dotsc,A(2m) \in M_N(\mc B)$ and $\pi, \sigma \in
NC(m)$. Let $E_N = \tr_N \otimes \mathrm{id}_{\mc B}$.  If $\sigma \leq \pi$, then
\begin{equation*}
 \sum_{\substack{1 \leq j_1,\dotsc,j_{2m} \leq N\\ \widetilde \sigma \leq \ker \mathbf j}} \sum_{\substack{1 \leq i_1,\dotsc,i_{2m} \leq N\\ \widetilde{K(\pi)} \leq \ker \mathbf i}} A(1)_{j_1j_2}A(2)_{i_1i_2}\dotsb A(2m)_{i_{2m-1}i_{2m}} = N^{|\sigma|+|K(\pi)|}E_{N}^{(\sigma \wr K(\pi))}[A(1),\dotsc,A(2m)].
\end{equation*}
\end{prop}

\begin{proof}
First observe that the sum above can be rewritten as
\begin{equation*}
 \sum_{\substack{1 \leq i_1,\dotsc,i_{4m} \leq N\\ \widetilde{\sigma \wr K(\pi)} \leq \ker \mathbf i}} A(1)_{i_1i_2}\dotsb A(2m)_{i_{4m-1}i_{4m}}.
\end{equation*}
So this will follow from the formula
\begin{equation*}
 \sum_{\substack{1 \leq i_1,\dotsc,i_{2m} \leq N \\ \widetilde \sigma \leq \ker \mathbf i}} A(1)_{i_1i_2}\dotsb A(m)_{i_{2m-1}i_{2m}} = N^{|\sigma|} E_N^{(\sigma)}[A(1),\dotsc,A(m)]
\end{equation*}
for any $\sigma \in NC(m)$.

We will prove this by induction on the number of blocks of $m$.  If $\sigma = 1_m$ has
only one block, then we have
\begin{align*}
 \sum_{\substack{1 \leq i_1,\dotsc,i_{2m} \leq N \\ \widetilde \sigma \leq \ker \mathbf i}} A(1)_{i_1i_2}\dotsb A(m)_{i_{2m-1}i_{2m}}  &= \sum_{\substack{1 \leq i_1,\dotsc,i_m \leq N}} A(1)_{i_1i_2}A(2)_{i_2i_3}\dotsb A(m)_{i_mi_1}\\
&= N\cdot E_N(A(1)\dotsb A(m)).
\end{align*}

Suppose now that $V = \{l+1,\dotsc,l+s\}$ is an interval of $\sigma$.  Then
\begin{multline*}
 \sum_{\substack{1 \leq i_1,\dotsc,i_{2m} \leq N \\ \widetilde \sigma \leq \ker \mathbf i}} A(1)_{i_1i_2}\dotsb A(m)_{i_{2m-1}i_{2m}} \\
= \sum_{\substack{1 \leq i_1,\dotsc,i_{2l-2},\\
i_{2(l+s)+1},\dotsc,i_{2m} \leq N \\ \widetilde{\sigma \setminus V} \leq \ker \mathbf i}}A(1)_{i_1i_2}\dotsb \biggl(\sum_{1 \leq j_1,\dotsc,j_s \leq N} A(l+1)_{j_1j_2}\dotsb A(l+s)_{j_sj_1}\biggr)\dotsb A(m)_{i_{2m-1}i_{2m}}\\
=\sum_{\substack{1 \leq i_1,\dotsc,i_{2l-2},i_{2(l+s)+1},\dotsc,i_{2m} \leq N \\
\widetilde{\sigma \setminus V} \leq \ker \mathbf i}}A(1)_{i_1i_2}\dotsb \bigl(N\cdot
E_N(A(l+1)\dotsb A(l+s))\bigr)\dotsb A(m)_{i_{2m-1}i_{2m}},
\end{multline*}
which by induction is equal to
\begin{equation*}
N^{|\sigma|}E_N^{(\sigma \setminus V)}[A(1),\dotsc,A(l)E_N(A(l+1)\dotsb
A(l+s)),\dotsc,A(m)] = N^{|\sigma|}E_N^{(\sigma)}[A(1),\dotsc,A(m)].
\end{equation*}
\end{proof}

\begin{rmk}
We will also need to control the sum appearing in the proposition above for $\sigma,\pi
\in NC(m)$ with $\sigma \not\leq \pi$.  If $\mc B$ is commutative this poses no
difficulty, as then
\begin{multline*}
 \sum_{\substack{1 \leq j_1,\dotsc,j_{2m} \leq N\\ \widetilde \sigma \leq \ker \mathbf j}} \sum_{\substack{1 \leq i_1,\dotsc,i_{2m} \leq N\\ \widetilde{K(\pi)} \leq \ker \mathbf i}} A(1)_{j_1j_2}A(2)_{i_1i_2}\dotsb A(2m)_{i_{2m-1}i_{2m}} \\
=  \biggl(\sum_{\substack{1 \leq j_1,\dotsc,j_{2m} \leq N\\ \widetilde \sigma \leq \ker \mathbf j}}  A(1)_{j_1j_2}\dotsb A(2m-1)_{j_{2m-1}j_{2m}} \biggr)\biggl(\sum_{\substack{1 \leq i_1,\dotsc,i_{2m} \leq N\\ \widetilde{K(\pi)} \leq \ker \mathbf i}}A(2)_{i_1i_2}\dotsb A(2m)_{i_{2m-1}i_{2m}}\biggr)\\
=N^{|\sigma|+|K(\pi)|}
E_N^{(\sigma)}[A(1),\dotsc,A(2m-1)]E_N^{(K(\pi))}[A(2),\dotsc,A(2m)].
\end{multline*}
However, when $\mc B$ is noncommutative it is not clear how to express this sum in terms
of expectation functionals.  Instead, we will use the following bound on the norm:
\end{rmk}

\begin{prop}\label{normbound}
Let $\mc B$ be a unital C$^*$-algebra, and let $A(1),\dotsc,A(2m) \in M_N(\mc B)$.  If
$\sigma,\pi \in NC(m)$ then
\begin{equation*}
 \biggl\| \sum_{\substack{1 \leq j_1,\dotsc,j_{2m} \leq N\\ \widetilde \sigma \leq \ker \mathbf j}} \sum_{\substack{1 \leq i_1,\dotsc,i_{2m} \leq N\\ \widetilde{K(\pi)} \leq \ker \mathbf i}} A(1)_{j_1j_2}A(2)_{i_1i_2}\dotsb A(2m)_{i_{2m-1}i_{2m}} \biggr\| \leq N^{|\sigma|+|K(\pi)|}\|A(1)\|\dotsb \|A(2m)\|.
\end{equation*}
\end{prop}

\begin{proof}
For this proof, we extend the definition of $\widetilde \pi$ to all partitions $\pi \in
\mc P(m)$ in the obvious manner. We can rewrite expression above as
\begin{equation*}
\sum_{\substack{1 \leq i_1,\dotsc,i_{4m} \leq N\\ \widetilde{\sigma \wr K(\pi)} \leq \ker
\mathbf i}} A(1)_{i_1i_2}\dotsb A(2m)_{i_{4m-1}i_{4m}},
\end{equation*}
and so the result will follow from
\begin{equation*}
 \biggl\|\sum_{\substack{1 \leq i_1,\dotsc,i_{2m} \leq N\\ \widetilde \sigma \leq \ker \mathbf i}} A(1)_{i_1i_2}\dotsb A(m)_{i_{2m-1}i_{2m}}\biggr\| \leq N^{|\sigma|}\|A(1)\|\dotsb \|A(m)\|
\end{equation*}
for any partition $\sigma \in \mc P(m)$.

The idea now is to realize this expression as the trace of a larger matrix.  For each $V
\in \sigma$, let $M_N^V$ be a copy of $M_N(\C)$.  Consider the algebra
\begin{equation*}
 \bigotimes_{V \in \sigma} M_N^V \simeq M_{N^{|\sigma|}}(\C),
\end{equation*}
with the natural unital inclusions $\iota_V$ of $M_N^V$ for $V \in \sigma$.  For $1 \leq
l \leq m$, let
\begin{equation*}
 X(l) = \bigl(\iota_{\sigma(l)} \otimes \mathrm{id}_{\mc B}\bigr) A(l) \in \biggl(\bigotimes_{V \in \sigma} M_N^V\biggr) \otimes \mc B \simeq M_{N^{|\sigma|}}(\mc B),
\end{equation*}
where we have used the notation $\sigma(l)$ for the block of $\sigma$ which contains $l$.

In other words, $X(l)$ is the matrix indexed by maps $i:\sigma \to [N]= \{1,\dotsc,N\}$
such that
\begin{equation*}
 X(l)_{ij} = A(l)_{i(\sigma(l))j(\sigma(l))} \prod_{\substack{V \in \sigma\\ l \notin V}} \delta_{i(V)j(V)}.
\end{equation*}

Consider now the trace
\begin{align*}
 (\Tr_{N^{|\sigma|}} \otimes \mathrm{id}_{\mc B})(X(1)\dotsb X(m)) &= \sum_{\substack{i_1,\dotsc,i_m\\ i_l:\sigma \to [N]}} X(1)_{i_1i_2}\dotsb X(m)_{i_mi_1}\\
&= \sum_{\substack{i_1,\dotsc,i_m\\ i_l:\sigma \to [N]}}
A(1)_{i_1(\sigma(1))i_2(\sigma(1))}\dotsb A(m)_{i_m(\sigma(m))i_1(\sigma(m))} \prod_{1
\leq l \leq m} \prod_{\substack{V \in \sigma\\ l \notin V}}
\delta_{i_l(V)i_{\gamma(l)}(V)},
\end{align*}
where $\gamma \in S_m$ is the cyclic permutation $(123\dotsb m)$.  The nonzero terms in
this sum are obtained as follows:  For each block $V = (l_1 < \dotsb < l_s)$ of $\sigma$,
choose $1 \leq i_{l_1}(V),i_{\gamma(l_1)}(V),\dotsc,i_{l_s}(V),i_{\gamma(l_s)}(V) \leq N$
with the restrictions $i_{\gamma(l_1)}(V) = i_{l_2}(V),\dotsc,i_{\gamma(l_{s-1})}(V) =
i_{l_s}(V)$ and $i_{\gamma(l_s)}(V) = i_{l_1}(V)$.  Comparing with the definition of
$\widetilde \sigma$, it follows that
\begin{equation*}
 (\Tr_{N^{|\sigma|}} \otimes \mathrm{id}_{\mc B})(X(1)\dotsb X(m)) =  \sum_{\substack{1 \leq i_1,\dotsc,i_{2m} \leq N \\ \widetilde \sigma \leq \ker \mathbf i}} A(1)_{i_1i_2}\dotsb A(m)_{i_{2m-1}i_{2m}}
\end{equation*}
is the expression to be bounded.  However, $(\tr_{N^{|\sigma|}} \otimes \mathrm{id}_{\mc
B}) = N^{-|\sigma|}(\Tr_{N^{|\sigma|}} \otimes \mathrm{id}_{\mc B})$ is a contractive
conditional expectation onto $\mc B$ and so
\begin{equation*}
 \|(\Tr_{N^{|\sigma|}} \otimes \mathrm{id}_{\mc B})(X(1)\dotsb X(m))\| \leq N^{|\sigma|}\|X(1)\|\dotsb\|X(m)\|.
\end{equation*}
Since $(\iota_V \otimes \mathrm{id}_{\mc B})$ is a contractive $*$-homomorphism, we have
$\|X(l)\| = \|(\iota_{\sigma(l)} \otimes \mathrm{id}_{\mc B})(A(l))\| \leq \|A(l)\|$ and
the result follows.
\end{proof}

\begin{rmk*}
We are now prepared to prove the main theorem.
\end{rmk*}

\begin{proof}[Proof of Theorem \ref{mainthm}]
Fix $p_1,\dotsc,p_{2m} \in \mc B \langle t(i)| i \in I \rangle$, and set $A_N(k) =
p_k(D_N)$ for $1 \leq k \leq 2m$.  For notational simplicity, we will suppress the
subscript $N$ in our computations.

Let $l_1,\dotsc,l_{2m} \in \N$, $\epsilon_1,\dotsc,\epsilon_{2m} \in \{1,*\}$ and
consider
\begin{multline*}
(\psi_N^{*\infty} \otimes E_N)[U(l_1)^{\epsilon_1}A(1)U(l_2)^{\epsilon_2}\dotsb U(l_{2m})^{\epsilon_{2m}}A(2m)]\\
= (\psi_N^{*\infty} \otimes \mathrm{id}_{\mc B}) N^{-1}\sum_{\substack{1 \leq i_1,\dotsc,i_{4m} \leq N}} (U(l_1)^{\epsilon_1})_{i_1i_2}A(1)_{i_2i_3}(U(l_2)^{\epsilon_2})_{i_3i_4}\dotsb A(2m)_{i_{4m}i_1}\\
= \sum_{1 \leq i_1,\dotsc,i_{4m} \leq N}
N^{-1}\psi_N^{*\infty}\bigl[(U(l_1)^{\epsilon_1})_{i_1i_2}\dotsb
(U(l_{2m})^{\epsilon_{2m}})_{i_{4m-1}i_{4m}}\bigr]A(1)_{i_2i_3}\dotsb A(2m)_{i_{4mi_1}}.
\end{multline*}
By Corollary \ref{freeadjint}, this is equal to
\begin{equation*}
\sum_{1 \leq i_1,\dotsc,i_{4m} \leq N}
N^{-1}\psi_N^{*\infty}\bigl[U(l_1)^{\epsilon_1}_{i_1i_2}U(l_2)^{\epsilon_2}_{i_4i_3}\dotsb
U(l_{2m})^{\epsilon_{2m}}_{i_{4m}i_{4m-1}}\bigr]A(1)_{i_2i_3}\dotsb A(2m)_{i_{4mi_1}}.
\end{equation*}
After reindexing, this becomes
\begin{equation*}
 \sum_{\substack{1 \leq i_1,\dotsc,i_{2m} \leq N}} \sum_{1 \leq j_1,\dotsc,j_{2m} \leq N} \negthickspace \negthickspace \negthickspace \negthickspace N^{-1}\psi_N^{*\infty}\bigl[U(l_1)^{\epsilon_1}_{i_{2m}j_1}U(l_2)^{\epsilon_2}_{i_1j_2}\dotsb U(l_{2m})^{\epsilon_{2m}}_{i_{2m-1}j_{2m}}\bigr]A(1)_{j_1j_2}A(2)_{i_1i_2}\dotsb A(2m)_{i_{2m-1}i_{2m}}.
\end{equation*}
Applying Corollary \ref{integrate}, we have
\begin{multline*}
 \sum_{\substack{1 \leq i_1,\dotsc,i_{2m} \leq N}} \sum_{1 \leq j_1,\dotsc,j_{2m} \leq N} \sum_{\substack{\pi,\sigma \in NC^{\epsilon}(m)\\ \widetilde \pi \leq \overrightarrow{\ker \mathbf i} \wedge \ker \mathbf l\\ \widetilde \sigma \leq \ker \mathbf j \wedge \ker \mathbf l}} \negthickspace \negthickspace \negthickspace N^{-|K(\pi)|-|\sigma|}(\mu_{m}(\sigma,\pi) + O(N^{-2}))A(1)_{j_1j_2}A(2)_{i_1i_2}\dotsb A(2m)_{i_{2m-1}i_{2m}}\\
= \negthickspace \negthickspace\sum_{\substack{\pi,\sigma \in NC^{\epsilon}(m)\\
\widetilde \pi \leq \ker \mathbf l\\ \widetilde \sigma \leq \ker \mathbf l}}
\negthickspace (\mu_m(\sigma,\pi) + O(N^{-2}))N^{-|K(\pi)|-|\sigma|}\negthickspace
\negthickspace \negthickspace \sum_{\substack{1 \leq j_1,\dotsc,j_{2m} \leq N\\
\widetilde \sigma \leq \ker \mathbf j}}\sum_{\substack{1 \leq i_1,\dotsc,i_{2m} \leq N\\
\widetilde{K(\pi)} \leq \ker \mathbf i}} A(1)_{j_1j_2}A(2)_{i_1i_2}\dotsb
A(2m)_{i_{2m-1}i_{2m}}.
\end{multline*}
By Propositions \ref{expfunct} and \ref{normbound}, this is equal to
\begin{equation*}
 \sum_{\substack{\pi,\sigma \in NC^{\epsilon}(m)\\ \sigma \leq \pi\\ \widetilde \pi \vee \widetilde \sigma \leq \ker \mathbf l}} \mu_m(\sigma,\pi)E_N^{(\sigma \wr K(\pi))}[A(1),\dotsc,A(2m)],
\end{equation*}
up to $O(N^{-2})$ with respect to the norm on $\mc B$.  Set $a(k) = p_k(d_N)$ for $1 \leq k \leq 2m$,
then by Proposition \ref{limitform} we have
\begin{align*}
 E[u(l_1)^{\epsilon_1}a(1)\dotsb u(l_{2m})^{\epsilon_{2m}}a(2m)] &=  \sum_{\substack{\pi,\sigma \in NC^{\epsilon}(m)\\ \sigma \leq \pi\\ \widetilde \pi \vee \widetilde \sigma \leq \ker \mathbf l}} \mu_m(\sigma,\pi)E^{(\sigma \wr K(\pi))}[a(1),\dotsc,a(2m)]\\
&=  \sum_{\substack{\pi,\sigma \in NC^{\epsilon}(m)\\ \sigma \leq \pi\\ \widetilde \pi
\vee \widetilde \sigma \leq \ker \mathbf l}} \mu_m(\sigma,\pi)E_N^{(\sigma \wr
K(\pi))}[A(1),\dotsc,A(2m)],
\end{align*}
and the result now follows immediately.
\end{proof}

\begin{rmk*}\textbf{Randomly quantum rotated matrices.}  It follows easily from Theorem \ref{mainthm2} and the definition of asymptotic freeness that under the hypotheses of the theorem, the sets
\begin{equation*}
 (\{D_N(i): i \in I\}, \{U_N(1)D_N(i)U_N(1)^*: i \in I\}, \{U_N(2)D_N(i)U_N(2)^*: i \in I\},\dotsc )
\end{equation*}
are asymptotically (infinitesimally) free with amalgamation over $\mc B$ as $N \to
\infty$.  The condition on existence of a limiting joint distribution can be weakened
slightly as follows:
\end{rmk*}

\begin{cor}\label{qrotate}
Let $\mc B$ be a unital C$^*$-algebra, and let $(D_N(i))_{i \in I}$ and $(D'_N(j))_{j \in
J}$ be two families of matrices in $M_N(\mc B)$ for $N \in \N$.  Suppose that there is a
finite constant $C$ such that $\|D_N(i)\| \leq C$ and $\|D'_N(j)\| \leq C$ for $N \in
\N$, $i \in I$ and $j \in J$.  For each $N \in \N$, let $U_N$ be a $N \times N$ Haar
quantum unitary random matrix, independent from $\mc B$.
\begin{enumerate}
 \item If the joint distributions of $(D_N(i))_{i \in I}$ and $(D'_N(j))_{j \in J}$ both converge weakly (in norm or in $L^2$ with respect to a faithful trace), then  $(U_ND_N(i)U_N^*)_{i \in I}$ and $(D'_N(j))_{j \in J}$ are asymptotically free with amalgamation over $\mc B$ as $N \to \infty$.
 \item If the joint distribution of $(D_N(i))_{i \in I}$ and $(D'_N(j))_{j \in J}$ both converge infinitesimally (in norm or in $L^2$ with respect to a faithful trace), then $(U_ND_N(i)U_N^*)_{i \in I}$ and $(D'_N(j))_{j \in J}$ are asymptotically infinitesimally free with amalgamation over $\mc B$.
\end{enumerate}

\end{cor}

\begin{proof}
The only condition of Theorem \ref{mainthm2} which is not satisfied is that $\{D_N(i): i
\in I\} \cup \{D'_N(j): j \in J\}$ should have a limiting (infinitesimal) joint
distribution as $N \to \infty$.  We can see that this is not an issue as follows.  Let
$p_1,\dotsc,p_m \in B \langle t(i)| i \in I\rangle$ and $q_1,\dotsc,qm \in B \langle
t(j)| j \in J \rangle$ and set $A_N(k) = p_k(D_N)$, $B_N(k) = q_k(D'_N)$ for $1 \leq k
\leq m$.  From the proof of Theorem \ref{mainthm}, we have
\begin{equation*}
 (\psi_N \otimes E_N)[UA(1)U^*B(1)\dotsb UA(m)U^*B(m)] = \sum_{\substack{\pi,\sigma \in NC(m)\\ \sigma \leq \pi}} \mu_m(\sigma,\pi) E_N^{(\sigma \wr K(\pi))}[A(1),B(1),\dotsc,A(m),B(m)],
\end{equation*}
up to $O(N^{-2})$.  But the right hand side depends only on the distributions of
$(D(i))_{i \in I}$ and $(D'(j))_{j \in J}$, and so the result follows from Theorem
\ref{mainthm2}.
\end{proof}

\begin{rmk*} \textbf{Classical Haar unitary random matrices.}  In the remainder of this section, we will discuss the failure of these results for classical Haar unitaries.  First we show that if $\mc B$ is finite dimensional, then classical Haar unitaries are sufficient.
\end{rmk*}

\begin{prop}\label{finitedim}
Let $\mc B$ be a finite dimensional C$^*$-algebra, and let $(D_N(i))_{i \in I}$ be a
family of matrices in $M_N(\mc B)$ for each $N \in \N$.  Assume that there is a finite
constant $C$ such that $\|D_N(i)\| \leq C$ for all $N \in \N$ and $i \in I$.  For each $N
\in \N$, let $(U_N(l))_{l \in \N}$ be a family of independent $N \times N$ Haar unitary
random matrices, independent from $\mc B$.  Let $(u(l),u(l)^*)_{l \in \N}$ and
$(d_N(i))_{i \in I, N \in \N}$ be random variables in a $\mc B$-valued probability space
$(\mc A, E: \mc A \to \mc B)$ such that
\begin{enumerate}
\item $(u(l),u(l)^*)_{l \in \N}$ is free from $(d_N(i))_{i \in I}$ with respect to $E$ for each $N \in \N$.
 \item $(\{u(l),u(l)^*\})_{l \in \N}$ is a free family with respect to $E$, and $u(l)$ is a Haar unitary, independent from $\mc B$ for each $l \in \N$.
\item $(d_N(i))_{i \in I}$ has the same $\mc B$-valued joint distribution with respect to $E$ as $(D_N(i))_{i \in I}$ has with respect to $E_N$.

\end{enumerate}
Then for any polynomials $p_1,\dotsc,p_{2m} \in \mc B \langle t(i): i \in I \rangle$,
$l_1,\dotsc,l_{2m} \in \N$ and $\epsilon_1,\dotsc,\epsilon_{2m} \in \{1,*\}$,
\begin{equation*}
 \bigl\|(\psi_N^{*\infty} \otimes E_N)[U_N(l_1)^{\epsilon_1}p_1(D_N)\dotsb U_N(l_{2m})^{\epsilon_{2m}}p_{2m}(D_N)] - E[u(l_1)^{\epsilon_1}p_1(d_N)\dotsb u(l_{2m})^{\epsilon_{2m}}p_{2m}(d_N)]\bigr\|
\end{equation*}
is $O(N^{-2})$ as $N \to \infty$.

\end{prop}

\begin{proof}
Let $e_1,\dotsc,e_q$ be a basis for $\mc B$ with $\|e_r\| = 1$ for $1 \leq r \leq q$.
Let $p_1,\dotsc,p_{2m} \in \mc B \langle t(i)| i \in I \rangle$, let $A_N(k) = p_k(D_N)$
and let $A_N(k,r) \in M_N(\C)$ be the matrix of coefficients of the entries of $A_N(k)$
on $e_r$ for $1 \leq k \leq 2m$ and $1 \leq r \leq q$.  Let $a_N(k,r)$ and
$(u(l),u(l)^*)_{l \in \N}$ be random variables in a noncommutative probability space
$(\mc A, \varphi)$ such that
\begin{enumerate}
\item $\{a_N(k,r): 1 \leq k \leq 2m, 1 \leq r \leq q\}$ and $(u(l),u(l)^*)_{l \in \N}$ are free with respect to $\varphi$.
 \item $(a_N(k,r))_{1 \leq k \leq 2m, 1 \leq r \leq q}$ has the same joint distribution with respect to $\varphi$ as $(A_N(k,r))_{1 \leq k \leq 2m, 1 \leq r \leq q}$ with respect to $\tr_N$.
 \item $(u(l),u(l)^*)_{l \in \N}$ are freely independent with respect to $\varphi$ and $u(l)$ has a Haar unitary distribution.
\end{enumerate}
For $1 \leq k \leq 2m$ and $N \in \N$, let $a_N(k) = \sum a_N(k,r) \otimes e_r \in \mc A
\otimes \mc B$, and note that the family $(a_n(k))_{1 \leq k \leq 2m}$ has the same joint
distribution with respect to $E = \varphi \otimes \mathrm{id}_{\mc B}$ as does
$(A_N(k))_{1 \leq k \leq 2m}$ with respect to $E_N$.  Identifying $u(l) = u(l) \otimes 1$
in $\mc A \otimes \mc B$, it is also easy to see that $(u(l),u(l)^*)$ and $(a_N(k))_{1
\leq k \leq 2m}$ are freely independent with respect to $E$.

Now let $\epsilon_1,\dotsc,\epsilon_{2m} \in \{1,*\}$ and consider
\begin{multline*}
 (\tr_N \otimes \mb E \otimes \mathrm{id}_{\mc B})[U(l_1)^{\epsilon_1}A(1)\dotsb A(2m) U(l_{2m})^{\epsilon_{2m}}]\\
= \sum_{1 \leq r_1,\dotsc,r_{2m} \leq q} (\tr_N \otimes \mb
E)[U(l_1)^{\epsilon_1}A(1,r_1)\dotsb A(2m,r_{2m})U(l_{2m})^{\epsilon_{2m}}]e_{r_1}\dotsb
e_{r_{2m}}.
\end{multline*}
Since $\|e_r\| = 1$, it follows that
\begin{multline*}
 \bigl\|(\tr_N \otimes \mb E \otimes \mathrm{id}_{\mc B})[U(l_1)^{\epsilon_1}A(1)\dotsb U(l_{2m})^{\epsilon_{2m}}A(2m)] - E[u(l_1)^{\epsilon_1}a(1)\dotsb u(l_{2m})^{\epsilon_{2m}}a(2m)]\bigr\|\\
\leq \sum_{1 \leq r_1,\dotsc,r_{2m} \leq q} \bigl|(\tr_N \otimes \mb E)[U(l_1)^{\epsilon_1}A(1,r_1)\dotsb U(l_{2m})^{\epsilon_{2m}}A(2m,r_{2m})]\\
 - \varphi[u(l_1)^{\epsilon_1}a(1,r_1)\dotsb u(l_{2m})^{\epsilon_{2m}}a(2m,r_{2m})]\bigr|.
\end{multline*}
From standard asymptotic freeness results (see e.g. \cite{col1}), this expression is
$O(N^{-2})$ as $N \to \infty$.
\end{proof}

\begin{rmk*}  We will now give an example to show that Theorem \ref{mainthm} may fail for classical Haar unitaries if the algebra $\mc B$ is infinite dimensional.  First we recall the Weingarten formula for computing the expectation of a word in the entries of a $N \times N$ Haar unitary random matrix and its conjugate:
\begin{equation*}
 \mb E[U^{\epsilon_1}_{i_1j_1}\dotsb U^{\epsilon_{2m}}_{i_{2m}j_{2m}}] = \sum_{\substack{\pi,\sigma \in \mc P_2^{\epsilon}(2m)\\ \pi \leq \ker \mathbf i\\ \sigma \leq \ker \mathbf j}} W^c_{\epsilon N}(\pi,\sigma),
\end{equation*}
where $\mc P_2^{\epsilon}(2m)$ is the set of pair partitions for which each pairing
connects a $1$ with a $*$ in the string $\epsilon_1,\dotsc,\epsilon_{2m}$,  and
$W^c_{\epsilon N}$ is the corresponding Weingarten matrix, see \cite{col1, bs1}.
\end{rmk*}

\begin{ex} \label{counterexample}
Let $\mc B$ be a unital C$^*$-algebra, and for each $N \in \N$ let $\{E_{ij}(N,l): 1 \leq
i,j \leq N, l = 1,2\}$ be two commuting systems of matrix units in $\mc B$, i.e.,
\begin{enumerate}
\item $E_{i_1j_1}(N,1)E_{i_2j_2}(N,2) = E_{i_2j_2}(N,2)E_{i_1j_1}(N,1)$ for $1 \leq i_1,j_1,i_2,j_2 \leq N$.
 \item $E_{ij}(N,l)^* = E_{ji}(N,l)$ for $1 \leq i,j \leq N$.
 \item $E_{ik_1}(N,l)E_{k_2j}(N,l) = \delta_{k_1k_2}E_{ij}(N,l)$ for $1 \leq i,j,k_1,k_2 \leq N$.
 \item $E_{ii}(N,l)$ is a projection for $1 \leq i \leq N$, and
\begin{equation*}
 \sum_{i=1}^N E_{ii}(N,l) = 1.
\end{equation*}
\end{enumerate}
For $N \in \N$, define $A_N, B_N \in M_N(\mc B)$ by
\begin{align*}
 (A_N)_{ij} &= E_{ji}(N,1), & (B_N)_{ij} &= E_{ji}(N,2).
\end{align*}
Note that $A_N,B_N$ are self-adjoint and $A_N^2,B_N^2$ are the identity matrix, indeed
\begin{equation*}
 (A_N^2)_{ij} = \sum_{k=1}^N E_{ki}(N,1)E_{jk}(N,1) =  \delta_{ij} \sum_{k=1}^N E_{kk}(N,1) = \delta_{ij} \cdot 1,
\end{equation*}
and likewise for $B_N$.  It follows that $\|A_N\| = \|B_N\| = 1$ for $N \in \N$.

For each $N \in \N$, let $U_N$ be a $N \times N$ Haar unitary random matrix, independent
from $\mc B$.  Since
\begin{equation*}
 (\tr_N \otimes \mathrm{id}_{\mc B})[A_N] = \frac{1}{N}\sum_{i=1}^N E_{ii}(N,1) = \frac{1}{N} \cdot 1
\end{equation*}
converges to zero as $N \to \infty$, and likewise for $B_N$, for asymptotic freeness we
should have
\begin{equation*}
 \lim_{N \to \infty} (\tr_N \otimes \mb E \otimes \mathrm{id})[(U_NA_NU_N^*B_N)^3] = 0.
\end{equation*}
However, we will show that this limit is in fact equal to 1.

Indeed, suppressing the subindex $N$ we have
\begin{align*}
 (\tr \otimes \mb{E} \otimes \mathrm{id}_{\mc B})[(UAU^*B)^3] &= \frac{1}{N} \sum_{1 \leq i_1,\dotsc,i_{12} \leq N} \mb E[U_{i_1i_2}\overline{U}_{i_4i_3}\dotsb \overline{U}_{i_{12}i_{11}}] A_{i_2i_3}B_{i_4i_5}\dotsb B_{i_{12}i_1}\\
&= \sum_{\substack{1 \leq i_1,j_1,\dotsc,i_6,j_6 \leq N}} \mb
E[U_{i_6j_1}\overline{U}_{i_1j_2}\dotsb \overline{U}_{i_5j_6}]
A_{j_1j_2}A_{j_3j_4}A_{j_5j_6}B_{i_1i_2}B_{i_3i_4}B_{i_5i_6}.
\end{align*}
Applying the Weingarten formula, we obtain
\begin{equation*}
\sum_{\pi,\sigma \in \mc P_2^{\epsilon}(6)}N^{-1}W^c_{\epsilon
N}(\pi,\sigma)\bigl(\sum_{\substack{1 \leq j_1,\dotsc,j_{6} \leq N\\ \sigma \leq \ker
\mathbf j}}A_{j_1j_2}A_{j_3j_4}A_{j_5j_6} \biggr)\biggl(\sum_{\substack{1 \leq
i_1,\dotsc,i_6 \leq N\\ \overleftarrow{\pi} \leq \ker \mathbf i}}
B_{i_1i_2}B_{i_3i_4}B_{i_5i_6}\biggr).
\end{equation*}
Note that $\mc P_2^{\epsilon}(6)$ has 6 elements, namely the 5 noncrossing pair
partitions and $\tau = \{(1,4),(2,5),(3,6)\}$.  The noncrossing pair partitions can be
expressed as $\widetilde \sigma$ for some $\sigma \in NC(3)$, in which case we have
\begin{equation*}
 \sum_{\substack{1 \leq j_1,\dotsc,j_6 \leq N\\ \widetilde \sigma \leq \ker \mathbf j}} A_{j_1j_2}A_{j_3j_4}A_{j_5j_6} = N^{|\sigma|}E_N^{(\sigma)}[A,A,A].
\end{equation*}
Using $E_N[A] = E_N[A^3] = N^{-1}$ and $E_N[A^2] = 1$, one easily sees that this
expression is $O(N)$ for the 5 noncrossing pair partitions.  For $\tau$, we have
\begin{align*}
 \sum_{\substack{1 \leq j_1,\dotsc,j_6 \leq N\\ \tau \leq \ker \mathbf j}} A_{j_1j_2}A_{j_3j_4}A_{j_5j_6} &= \sum_{1 \leq j_1,j_2,j_3 \leq N} A_{j_1j_2}A_{j_3j_1}A_{j_2j_3}\\
&= \sum_{1 \leq j_1,j_2,j_3 \leq N} E_{j_2j_1}(N,1)E_{j_1j_3}(N,1)E_{j_3j_2}(N,1)\\
&= \sum_{1 \leq j_1,j_2,j_3 \leq N} E_{j_2j_2}(N,1)\\
&= N^2 \cdot 1,
\end{align*}
and likewise for $B_N$.  Also we have $N^{3}W^c_{\epsilon N}(\pi,\sigma) = \delta_{\pi
\sigma} + O(N^{-1})$.  Putting these statements together, we find that the only term
which remains in the limit comes from $\pi = \sigma = \tau$, which gives 1.
\end{ex}

\begin{rmk*} \textit{Remarks}.
\begin{enumerate}
 \item We note that $M_{N^2}(\C) = M_N(\C) \otimes M_N(\C)$ has a natural pair of commuting systems of matrix units, so this example demonstrates that Theorem \ref{mainthm} fails for any unital C$^*$-algebra $\mc B$ which contains $M_{N_k^2}(\C)$ as a unital subalgebra for some increasing sequence of natural numbers $(N_k)$.
\item It is a natural question whether the matrices $A_N,B_N$ in the above example have limiting $\mc B$-valued distributions, which would demonstrate that Theorem \ref{easythm} also fails for classical Haar unitaries.  First observe that
\begin{equation*}
 \lim_{N \to \infty} (\tr_N \otimes \mathrm{\mc B})[A_N^k] = \begin{cases} 1, & \text{$k$ is even}\\ 0, & \text{$k$ is odd} \end{cases},
\end{equation*}
which follows from the case $k = 1$ and the fact that $A_N^2$ is the identity matrix.
However, it is not clear that moments of the form $b_0A_N\dotsb A_Nb_k$ will converge for
arbitrary $b_0,\dotsc,b_k \in \mc B$.

 Let us point out a special case in which the limiting distribution does exist.  Suppose that there is a dense $*$-subalgebra $\mc F \subset \mc B$ such that each element of $\mc F$ commutes with the matrix units $E_{ij}(N,l)$ for $N$ sufficiently large.  Then for any $b_0,\dotsc,b_k \in \mc B$ we have
\begin{equation*}
  \lim_{N \to \infty} (\tr_N \otimes \mathrm{\mc B})[b_0A_N\dotsb A_Nb_k] = \begin{cases} b_0b_1\dotsb b_k, & \text{$k$ is even}\\ 0, & \text{$k$ is odd} \end{cases},
\end{equation*}
and likewise for $B_N$, indeed this holds for $b_0,\dotsc,b_k \in \mc F$ by hypothesis
and for general $b_0,\dotsc,b_k$ by density.

In particular, we may take $\mc B$ to be the C$^*$-algebraic infinite tensor product
\begin{equation*}
 \mc B = \bigotimes_{N \in \N} M_N(\C)
\end{equation*}
with the obvious systems of matrix units $E(N,l)_{ij} \in M_{N^2} = M_N(\C) \otimes
M_N(\C) \subset \mc B$, and $\mc F \subset \mc B$ to be the image of the purely algebraic
tensor product.  Note that $\mc B$ is \textit{uniformly hyperfinite}, in particular
\textit{approximately finitely dimensional} in the C$^*$-sense.

\item Note that if $\mc B$ is a von Neumann algebra with a non-zero \textit{continuous} projection $p$, then $p\mc B p$ contains $M_N(\C)$ as a unital subalgebra for all $N \in \N$ and hence (1) applies to $p\mc B p$.  It follows that Theorem \ref{mainthm} fails also for $\mc B$.  To obtain a contradiction to Theorem \ref{easythm} for classical Haar unitaries in the setting of a von Neumann algebra with faithful, normal trace, we may modify the example in (2) by taking $(\mc B,\tau)$ to be the infinite tensor product
\begin{equation*}
 (\mc B, \tau) = \bigotimes_{N \in \N} (M_N(\C),\tr_N)
\end{equation*}
taken with respect to the trace states $\tr_N$ on $M_N(\C)$, which is the
\textit{hyperfinite $II_1$ factor}.
\end{enumerate}

\end{rmk*}

\end{document}